\documentclass[reqno,11pt]{amsart}

\usepackage[a4paper,left=25mm,right=25mm,top=30mm,bottom=30mm,marginpar=25mm]{geometry} 
\usepackage{amsmath}
\usepackage{amssymb}
\usepackage{amsthm}
\usepackage{amscd}
\usepackage[ansinew]{inputenc}
\usepackage{color}
\usepackage[final]{graphicx}
\usepackage{tikz}
\usepackage{bbm}

\usetikzlibrary{patterns} 
\usetikzlibrary{positioning}
\usetikzlibrary{calc}
\usetikzlibrary{arrows,shapes,backgrounds}
\usetikzlibrary{intersections} 

\renewcommand{\epsilon}{\varepsilon}

\numberwithin{equation}{section}

\newtheoremstyle{thmlemcorr}{10pt}{10pt}{\itshape}{}{\bfseries}{.}{10pt}{{\thmname{#1}\thmnumber{ #2}\thmnote{ (#3)}}}
\newtheoremstyle{thmlemcorr*}{10pt}{10pt}{\itshape}{}{\bfseries}{.}\newline{{\thmname{#1}\thmnumber{ #2}\thmnote{ (#3)}}}
\newtheoremstyle{defi}{10pt}{10pt}{\itshape}{}{\bfseries}{.}{10pt}{{\thmname{#1}\thmnumber{ #2}\thmnote{ (#3)}}}
\newtheoremstyle{remexample}{10pt}{10pt}{}{}{\bfseries}{.}{10pt}{{\thmname{#1}\thmnumber{ #2}\thmnote{ (#3)}}}
\newtheoremstyle{ass}{10pt}{10pt}{}{}{\bfseries}{.}{10pt}{{\thmname{#1}\thmnumber{ A#2}\thmnote{ (#3)}}}

\theoremstyle{thmlemcorr}
\newtheorem{theorem}{Theorem}
\numberwithin{theorem}{section}
\newtheorem{lemma}[theorem]{Lemma}
\newtheorem{corollary}[theorem]{Corollary}
\newtheorem{proposition}[theorem]{Proposition}

\theoremstyle{thmlemcorr*}
\newtheorem{theorem*}{Theorem}
\newtheorem{lemma*}[theorem]{Lemma}
\newtheorem{corollary*}[theorem]{Corollary}
\newtheorem{proposition*}[theorem]{Proposition}
\newtheorem{problem*}[theorem]{Problem}
\newtheorem{conjecture*}[theorem]{Conjecture}

\theoremstyle{defi}

\theoremstyle{remexample}
\newtheorem{remark}[theorem]{Remark}

\theoremstyle{ass}

\newcommand{\Acal}{\mathcal{A}}

\newcommand{\Mcal}{\mathcal{M}}
\newcommand{\Ncal}{\mathcal{N}}

\newcommand{\Ucal}{Q'}

\newcommand{\Ibb}{\mathbb{I}}

\DeclareMathOperator{\curl}{curl}
\DeclareMathOperator{\dist}{dist}
\DeclareMathOperator{\rank}{rank}

\newcommand{\norm}[1]{\|#1\|}

\newcommand{\normB}[1]{\Bigl\|#1\Bigr\|}

\newcommand{\abs}[1]{|#1|}

\newcommand{\absb}[1]{\bigl|#1\bigr|}
\newcommand{\absB}[1]{\Bigl|#1\Bigr|}

\newcommand{\dd}{\;\mathrm{d}}

\newcommand{\N}{\mathbb{N}}
\newcommand{\R}{\mathbb{R}}

\newcommand{\Z}{\mathbb{Z}}

\newcommand{\weakly}{\rightharpoonup}
\newcommand{\weaklystar}{\overset{*}\rightharpoonup}

\newcommand{\eps}{\epsilon}
\newcommand{\ucal}{Q}

\newcommand{\sbullet}{\begin{picture}(1,1)(-0.5,-2)\circle*{2}\end{picture}}
\newcommand{\frarg}{\,\sbullet\,}

\newcommand{\Yrig}{Y_{\rm rig}}
\newcommand{\Ysoft}{Y_{\rm soft}}

\def\Xint#1{\mathchoice 
{\XXint\displaystyle\textstyle{#1}}%
{\XXint\textstyle\scriptstyle{#1}}%
{\XXint\scriptstyle\scriptscriptstyle{#1}}%
{\XXint\scriptscriptstyle\scriptscriptstyle{#1}}%
\!\int} 
\def\XXint#1#2#3{{\setbox0=\hbox{$#1{#2#3}{\int}$} 
\vcenter{\hbox{$#2#3$}}\kern-.5\wd0}} 
\def\dashint{\,\Xint-}

\title[Homogenization in single-slip finite elasto-plasticity]{Homogenization of layered materials with rigid components in single-slip finite crystal plasticity}

\author{Fabian Christowiak}
\address{Fakult\"at f\"ur Mathematik, Universit\"at Regensburg, 93040 Regensburg, Germany}
\email{Fabian.Christowiak@mathematik.uni-regensburg.de}

\author{Carolin Kreisbeck}
\address{Fakult\"at f\"ur Mathematik, Universit\"at Regensburg, 93040 Regensburg, Germany}
\email{Carolin.Kreisbeck@mathematik.uni-regensburg.de}

\begin{document}

 
\maketitle

 \begin{abstract}  
 \vspace{-12pt}   
We determine the effective behavior of a class of composites in finite-strain crystal plasticity, based on a variational model for materials made of fine parallel layers of two types. While one component is completely rigid in the sense that it admits only local rotations, the other one is softer featuring a single active slip system with linear self-hardening. As a main result, we obtain explicit homogenization formulas by means of $\Gamma$-convergence. Due to the anisotropic nature of the problem, the findings depend critically on the orientation of the slip direction relative to the layers, leading to three qualitatively different regimes that involve macroscopic shearing  and blocking effects. The technical difficulties in the proofs are rooted in the intrinsic rigidity of the model, which translates into a non-standard variational problem constraint by non-convex partial differential inclusions. The proof of the lower bound requires a careful analysis of the admissible microstructures and a new asymptotic rigidity result, whereas the construction of recovery sequences relies on nested laminates.  

   \color{black}                       
\vspace{8pt}

 \noindent\textsc{MSC (2010):} 49J45 (primary); 74Q05, 74C15
 
 \noindent\textsc{Keywords:} homogenization, $\Gamma$-convergence, composite materials, finite crystal plasticity.

 \noindent\textsc{Date:} \today.
 \end{abstract}

\section{Introduction}\label{sec:introduction}

The search for new materials with desirable mechanical properties 
is one of the key tasks in materials science. 
As suitable combinations of different materials may exceed their individual constituents with regard to important characteristics, like strength, stiffness or ductility, composites play an important role in material design, e.g.~\cite{Mil02, Jon98, VaM13}.
In this pursuit, the following question is of fundamental interest: 
Given the arrangement and geometry of the building blocks on a mesoscopic level, as well as the deformation mechanisms inside the homogeneous components, can we predict the macroscopic material response of a sample under some applied external load? 

By now there are various homogenization methods available that help to give answers. A substantial body of literature has emerged in materials science, engineering, and mathematics, see for instance~\cite{JKO94, Mil02} and the references therein, or more specifically,~\cite{Raa04, ScH13, FrG14, KoL99} for heterogeneous plastic materials, and~\cite{mal07, BoB02, MSB02}  for fiber-reinforced materials, 
and~\cite{ChC12, BrG95} for high-contrast composites, 
to mention just a few references.
\color{black}
A rigorous analytical approach that has proven successful for variational models based on energy minimization principles rests on the concept of $\Gamma$-convergence introduced by de Giorgi and Franzoni~\cite{DeG75, DeF75}. 
By letting the length scale of the heterogeneities tend towards zero, one passes to a limit energy, which gives rise to the effective material model. 
 \color{black}
In this paper, we follow along these lines and study a variational model for reinforced bilayered materials in the context of geometrically nonlinear plasticity. The model is set in the plane and we assume that the material consists of periodically alternating strips of rigid 
components and softer ones that can be deformed plastically by single-slip.
As this problem is highly anisotropic, considering the layered structure and the distinguished orientation of the slip system, there are interesting interactions to be observed.

Let $\Omega\subset\R^2$ be a bounded Lipschitz domain, modeling the reference configuration of an elastoplastic body in two space dimensions, and let $u:\Omega\to \R^2$ be a deformation field. For describing the periodic material heterogeneities, we take the unit cell $Y=[0,1)^2$, and define for $\lambda \in (0,1)$ the subsets 
\begin{align}\label{def:YrigYsoft}
\Ysoft =[0,1)\times[0, \lambda) \subset Y \qquad \text{and}\qquad \Yrig=Y\setminus \Ysoft,
\end{align} 
which correspond to the softer and rigid component, respectively.  
Throughout this paper, we identify the sets $\Yrig$ and $\Ysoft$ with their $Y$-periodic extensions to $\R^2$. To provide a measure for the length scale of the oscillations between the material components, we introduce the parameter $\epsilon>0$, which describes the thickness of two neighboring layers. With these notations, the sets $\eps\Yrig\cap \Omega$ and $\eps\Ysoft\cap \Omega$ refer to the stiff and softer layers. For an illustration of the geometric set-up see Figure \ref{fig: bilayered structure}.

Following the classic work by Kr\"oner and Lee~\cite{Lee69, Kro60} on finite-strain crystal plasticity, we use the multiplicative decomposition of the deformation gradient $\nabla u = F_{\rm e} F_{\rm p}$ as a fundamental assumption. Here, the elastic part $ F_{\rm e}$ describes local rotation and stretching of the crystal lattice, and the inelastic part $ F_{\rm p}$ captures local plastic deformations resulting from the movement of dislocations. Recent progress on a rigorous derivation of the above splitting as the continuum limit of micromechanically defined elastic and plastic components has been made in~\cite{ReC14, RSC16}.

In this model, proper elastic deformations are excluded by requiring $F_{\rm e}$ to be (locally) a rotation, i.e.~$F_{\rm e}\in SO(2)$ pointwise. 
This lack of elasticity makes the overall material fairly rigid. For the plastic part, we impose $ F_{\rm p} = \Ibb$ on $\eps\Yrig\cap \Omega$, reflecting that there is no plastic deformation in the stiff layers. In the softer layers $\eps\Ysoft\cap \Omega$, plastic glide can occur along one active slip system $(s, m)$ with slip direction $s\in \R^2$ with $\abs{s}=1$ and slip plane normal $m=s^\perp$, so that integration of the plastic flow rule yields $F_{\rm p} = \Ibb + \gamma s\otimes m$, where $\gamma\in \R$ corresponds to the amount of slip, for more details see~\cite[Section~2]{CoT05}. 
Altogether, we observe that the deformation gradient $\nabla u$ is restricted pointwise to the set
\begin{align*}
\Mcal_s&= \{F\in \mathbb{R}^{2\times 2}: F= R(\mathbb{I}+\gamma s\otimes m), R\in SO(2), \gamma \in \mathbb{R}
\} \\ &= \{F\in \R^{2\times 2}: \det F=1, \abs{Fs}=1\},
\end{align*} 
and in the stiff components even to $SO(2)$.

As regards relevant energy expressions, the latter entails that the energy density in the rigid layers is given by $W_{\rm rig}(F)=0$ if $F\in SO(2)$ and $W_{\rm rig}(F)=\infty$ otherwise in $\R^{2\times 2}$. 
Moreover, adopting the homogeneous single-slip model with linear self-harding introduced in~\cite{Con06} (cf.~also~\cite{CDK11}) gives rise to the condensed energy density in the softer layers 
\begin{align}\label{Wsoft}
W_{\rm soft}(F)=\begin{cases} \gamma^2=\abs{Fm}^2-1  & \text{if $F=R(\Ibb+\gamma s\otimes m)\in \Mcal_s$},\\ \infty & \text{otherwise}, \end{cases}
\qquad F\in \R^{2\times 2}.
\end{align}
We combine the energy contributions in the two components to obtain the heterogeneous density 
\begin{align}\label{def_W}
 W(y,F) = \mathbbm{1}_{\Yrig}(y)W_{\rm rig}(F) +  \mathbbm{1}_{\Ysoft}(y) W_{\rm soft}(F), \qquad  y \in \R^2,\  F \in \R^{2\times2},
\end{align}
which is periodic with respect to the unit cell $Y$ and reflects the bilayered structure of the material. Here, $\mathbbm{1}_U$ is the symbol for the characteristic function of a set $U\subset \R^2$.

\color{black}
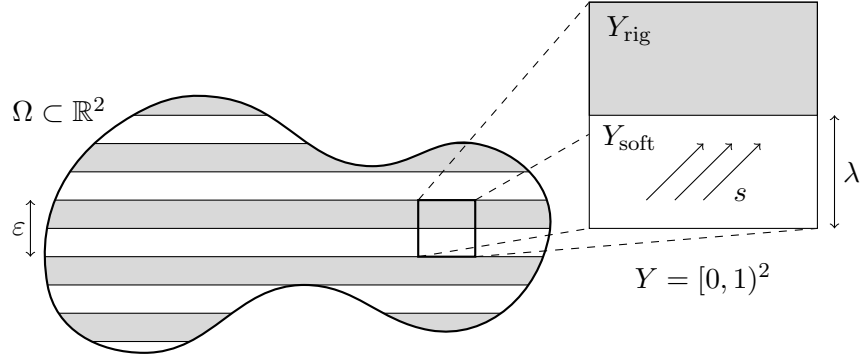
\begin{figure} \label{fig: bilayered structure}
\begin{tikzpicture}

\begin{scope}[scale = 1.5]
\clip(-0.7,-0) rectangle (7,3.3);
\def\material{(-0.5,1.5) to[out = 100, in=210] (1,4.5) to[out = 30, in=170]  (2.6,4.8) to[out = -10, in=180] (5.2,3.6) to[out = 0, in=170] (7, 4) to[out = -10, in=80] (8.3,2.4) to[out = 260, in=-10] (6.2, 0.7)[out = 0, in = 180] to[out = 170, in=0] (4, 1.5) to[out = 180, in=0] (1.2,0.3) to[out = 180, in=280] (-0.5,1.5)};

\def\materialflat{(0,2) to[out = 100, in=210] (1.1,4.2) to(3,4.2) to[out = -10, in=180] (5,3.7) to[out = 0, in=170] (7, 4) to[out = -10, in=80] (8,2) to[out = 260, in=10] (6, 0.5) to[out = 190, in=0] (4, 1) to[out = 180, in=40] (1.9,0) to (1.1,0) to[out = 140, in=280] (0,2)};

           \fill[white,fill opacity=.9] (0,0) rectangle (6,2);
           ,4.8
  	\foreach \y in {0,...,4}
        \draw[fill=gray!30] ($(-0.5,\y/2) + (0,0.25)$) rectangle ($(4.5,\y/2)+(0,0.5)$);
        
        \fill[even odd rule, fill = white,thick, scale = 0.5] \material (-3,0) rectangle (10,10);
        \draw[ thick, black ,fill opacity=.3, scale = 0.5] \material;
                
        \draw[thick, black] (3,1) rectangle (3.5,1.5);
        \draw[dashed,black] (3, 1.5) -- (4.5,3.25);
        \draw[dashed, black] (3.5,1) -- (6.5, 1.25);
        \draw[dashed, black] (3, 1) -- (4.5,1.25);
        \draw[dashed, black] (3.5,1.5) -- (6.5, 3.25);

        \draw[white, fill = white] (4.5,1.25) rectangle (6.5,3.25);
        \draw[rectangle, black, fill= gray!30] (4.5,2.25) rectangle (6.5,3.25);
        \draw[black] (4.5,1.25) rectangle (6.5,3.25);
        
        \draw (4.85,3) node {$\Yrig$};
	\draw (4.85,2.05) node {$\Ysoft$};
	\draw[<->] (6.65,1.25) -- (6.65,2.25);
	\draw (6.8,1.75) node {$\lambda$};

        \draw (-0.65,2.3) node[anchor = west] {$\Omega\subset \R^2$};
        \draw (5.5,0.8) node { $Y=[0,1)^2$};

	\draw (-0.65,1.25) node[anchor = west] {$\eps$};
	\draw[<->] (-0.4,1.0) -- (-0.4,1.5);
	\draw[->] (5, 1.5) -- (5.5, 2.0); 
	\draw[->] (5.25, 1.5) -- (5.75, 2.0);
	\draw[->] (5.5, 1.5) -- (6, 2.0);
	\draw (5.825, 1.55) node { $s$};	         
    \end{scope}    
    \end{tikzpicture} 
    \caption{Bilayered elastoplastic material with periodic structure; rigid components depicted in gray, softer components with one active slip system (slip direction $s$) in white.}
\end{figure}
According to~\cite{OrR99, CHM02}, the dynamical behavior of plastic materials under deformation can be well approximated by incremental minimization, that is by a time-discrete variational approach (for earlier work in the context of fracture and damage see~\cite{FrM93, FrM98}). Note that in this paper, we discuss only the first time step. This simplification suppresses delicate issues of microstructure evolution.  
As system energy of the first incremental problem we consider the energy functional $E_\eps: L^2_0(\Omega;\R^2)\to [0, \infty]$ for $\eps>0$ defined by 
\begin{align}\label{def:Eeps}
 E_\eps(u)=\int_\Omega W\left(\frac{x}{\eps}, \nabla u(x)\right)\dd{x}, \qquad u\in W^{1,2}(\Omega;\R^2)\cap L^2_0(\Omega;\R^2),
\end{align}
and $E_\eps(u)=\infty$ otherwise in $L^2_0(\Omega;\R^2)$, the space of $L^2$-functions with vanishing mean value. 
By~\eqref{def_W} and~\eqref{Wsoft}, one has the following equivalent representations of $E_\eps$,
\begin{equation}\label{Eeps_representation2}
 \begin{aligned}
E_\eps(u) &= \int_\Omega \gamma^2\dd{x} \quad \text{if~} u\in W^{1,2}(\Omega;\R^2), \nabla u=R(\Ibb+\gamma s\otimes m)  \\ &\hspace{2.7cm} \text{with~} R \in L^\infty(\Omega;SO(2)), \gamma \in L^2(\Omega), \gamma=0 \text{~a.e.~in~} \eps \Yrig\cap \Omega, 
\end{aligned}
\end{equation}
\begin{equation}\label{Eeps_representation1}
\begin{aligned}
&= \int_{\Omega} \abs{\nabla u m}^2-1\dd{x} \quad  \text{if $u\in W^{1,2}(\Omega;\R^2)$, $\nabla u\in \Mcal_s$ a.e.~in $\Omega$,  \hspace{1.6cm}} \\ &\hspace{4.2cm} \text{$\nabla u\in SO(2)$ a.e.~in $\eps\Yrig\cap \Omega$,}
\end{aligned}
\end{equation}
and $E_\eps(u)=\infty$ otherwise in $L^2_0(\Omega;\R^2)$. 
It  becomes apparent from~\eqref{Eeps_representation1} that the functionals $E_\eps$ are subject to non-convex constraints in the form of partial differential inclusions. Even though $E_\eps$ matches with an integral expression with quadratic integrand when finite, the constraints render the associated homogenization problem non-standard. 
In particular, it is not directly accessible to by now classical homogenization methods for variational integrals with quadratic growth as e.g.~in~\cite{Mul87, BrD98}. Due to the non-convexity of the sets $\Mcal_s$ and $SO(2)$, it does not fall within the scope of works on gradient-constraint problems like~\cite{CaDA02, CEYZ04, CDD06}, either. 

\color{black}

Our main result is the following theorem, which holds under the additional assumption that $\Omega$ is simply connected. It amounts to an explicit characterization of the $\Gamma$-limit of $(E_\eps)_\eps$ as $\eps$ tends to zero (for an introduction to $\Gamma$-convergence see e.g.~\cite{Dal93, Bra02}), and therefore, provides the desired homogenized model 
that describes the effective material response in the limit of vanishing layer thickness.

\begin{theorem}[Homogenization via $\Gamma$-convergence]\label{theo:Gammalimit}
The family $(E_\eps)_\eps$ $\Gamma$-converges to a functional $E:  L^2_0(\Omega;\R^2)\to [0, \infty]$ with respect to the strong $L^2(\Omega;\R^2)$-topology, in formulas, $\Gamma(L^2)\text{-}\lim_{\eps\to 0} E_\eps = E$, where $E$ is defined by
\begin{align}\label{def:GammalimitI}
E(u)=\begin{cases}\displaystyle\frac{s_1^2}{\lambda} \int_\Omega \gamma^2 \dd{x} - 2 s_1 s_2 \int_\Omega \gamma \dd x & \text{if $u\in W^{1,2}(\Omega;\R^2)$, 
$\nabla u=R(\Ibb +\gamma e_1\otimes e_2)$}  \text{~with} \\ &\quad \text{$R\in SO(2)$, $\gamma\in L^2(\Omega)$, 
$\gamma \in K_{s, \lambda}$ a.e.~in $\Omega$,}\\ \infty & \text{otherwise.} \end{cases}
\end{align} 
The pointwise restriction $K_{s,\lambda}$ for $s=(s_1, s_2)$ and $\lambda\in (0,1)$ is given by
\begin{align}\label{def_Ks}
K_{s, \lambda} = \begin{cases} \{0\} &\text{if~} s=e_2, \\ [-2\frac{s_1}{s_2}\lambda, 0] &\text{if~$ s_1 s_2 >0$,} \\ [0,-2\frac{s_1}{s_2}\lambda] &\text{if $s_1s_2 < 0$,} \\ \R & \text{if $s = e_1$}. \end{cases}
\end{align} 

Moreover, bounded energy sequences of $(E_\eps)_\eps$, i.e.~$(u_\eps)_\eps$ with $E_\eps(u_\eps)< C$ for all $\eps>0$, are relatively compact in $L^2_0(\Omega;\R^2)$.
\end{theorem}

Recalling the definition of $\Gamma$-convergence, Theorem~\ref{theo:Gammalimit} can be formulated in terms of these three statements:

\textit{
 (Compactness)} For $\eps_j\to 0$ and $(u_j)_j \subset L^2_0(\Omega;\R^2)$ with $E_{\eps_j}(u_{j})<C$ for all $j\in \N$, there exists a subsequence of $(u_j)_j$ (not relabeled) and $u\in L_0^2(\Omega;\R^2)$ with $E(u)<\infty$ 
such that $u_j\to u$ in $L^2(\Omega;\R^2)$.

\textit{
 (Lower bound)} Let $\eps_j\to 0$ and $(u_j)_j\subset L^2_0(\Omega;\R^2)$ with $u_j\to u$ in $L^2(\Omega;\R^2)$ for some $u\in L^2_0(\Omega;\R^2)$. Then,
\begin{align}\label{liminf}
\liminf_{j\to \infty} E_{\eps_j}(u_{j}) \geq E(u).
\end{align}

\textit{
(Recovery sequences)}
For every $u\in L_0^2(\Omega;\R^2)$ with $E(u)<\infty$, there exists $(u_\eps)_\eps\subset L_0^2(\Omega;\R^2)$ with $u_\eps\to u$ in $L^2(\Omega;\R^2)$ such that $\lim_{\eps \to 0} E_{\eps}(u_{\eps}) = E(u).$

\begin{remark}
In comparison with $E_\eps$, the differential constraints in the formulation of $E$ are substantially more restrictive, and cause the limit functional to  be essentially one-dimensional. 
While the gradients of finite-energy deformations for $E_\eps$ lie pointwise in the set $\Mcal_{s}$, those for $E$ take values in $\Mcal_{e_1}$, independent of $s$, and satisfy the additional restriction of a constant rotation. In particular, this implies that $\partial_1 \gamma = 0$, as gradient fields are curl-free. 
Notice also that the second term in $E$ is non-negative due to the pointwise restriction $\gamma\in K_{s, \lambda}$.
\end{remark}

\begin{remark}[Generalizations of Theorem~\ref{theo:Gammalimit}]
a) Except for only minor changes, the quadratic growth in the energies $E_\eps$ can be replaced by $p$-growth with $p\geq 2$. Calling the modified functionals $E_\eps^p$, we have that $E^p=\Gamma(L^p)\text{-}\lim_{\eps\to 0}E_\eps^p$ is characterized by
\begin{align*} 
E^p(u) = \int_{\Omega} \frac{1}{\lambda^{p-1}}|\nabla u m - (1-\lambda)Rm|^p - \lambda \dd{x},
\end{align*}
if $u\in W^{1,p}(\Omega;\R^2)$ such that $\nabla u=R(\Ibb +\gamma e_1\otimes e_2)$ with $R\in SO(2)$, $\gamma\in L^p(\Omega)$ and 
$\gamma \in K_{s, \lambda}$ a.e.~in $\Omega$, and $E^p(u)=\infty$ otherwise in $L^p_0(\Omega;\R^2)$.
For $p = 2$ this is a reformulation of \eqref{def:GammalimitI}. 

b) In the case $s=e_1$, we characterize the $\Gamma$-limit of the family $(E^\tau_\eps)_\eps$ defined in~\eqref{Eepskappa_representation2}, which results from $(E_\eps)_\eps$ by adding a linear dissipative term with prefactor $\tau\geq 0$. For the details see Section~\ref{sec:Gamma_e1}. 
We remark that this extension is motivated by~\cite{CoT05} and~\cite{Con06}. Whereas an explicit relaxation of 
the model involving the sum of a quadratic and linear expression is (to the best of our knowledge) unsolved, the homogenization result in the special case $s=e_1$ gets by without microstructure formation and can therefore manage the mixed expression.
\end{remark}

\color{black}
In the special cases, where the slip direction is parallel or orthogonal to the layered structure, the result of Theorem~\ref{theo:Gammalimit} reflects  some basic physical intuition. While for $s=e_2$ the effective body can only be rotated as a whole, as the rigid layers lead to a complete blocking of the slip system,  
the slip system is unimpeded if $s=e_1$, so that, macroscopically, (up to global rotations) exactly all shear deformations in horizontal direction can be achieved. 
If the slip direction is inclined, i.e.~$s\notin\{e_1, e_2\}$, the pointwise restriction $\gamma\in K_{s,\lambda}$ implies both that the effective horizontal shearing 
is only uni-directional (with the relevant direction depending on the orientation of $s$), which indicates a loss of symmetry, and that its maximum amount is capped.
In the limit energy, the factor $s_1^2/\lambda$ in front of the quadratic expression in $\gamma$ corresponds to an effective hardening modulus, recalling that $\lambda\in (0,1)$ stands for the relative thickness of the softer material layers.
For $s\notin\{e_1, e_2\}$, one observes (maybe surprisingly) an additional energy contribution that is linear in $\gamma$, which can be interpreted as a dissipative term.

Regarding the proof of Theorem~\ref{theo:Gammalimit}, we perform the usual three steps for $\Gamma$-convergence results by showing compactness and establishing matching upper and lower bounds.

The key to compactness and the lower bound is to capture the macroscopic effects of the bilayered material structure, which lead to an anisotropic reinforcement of the elastoplastic body.
In Proposition~\ref{prop:rigidity}, we establish a new type of asymptotic rigidity result, which is not specific to the context of plasticity, but potentially  applies to any kind of composite with rigid layers, provided that the macroscopic material response is a priori known to be volume-preserving.  
The reasoning relies on a well-known result by Reshetnyak (cf.~Lemma~\ref{lem:rigidity}), which implies that the stiff layers can only rotate as a whole, on an explicit estimate showing that rotations on neighboring rigid layers are close, and on a suitable one-dimensional compactness argument.
As a consequence of Proposition~\ref{prop:rigidity}, the weak limits of finite energy sequences for $(E_\eps)_\eps$ coincide necessarily with globally rotated shear deformations in $e_1$-direction. Gradients of the latter have the form $R(\Ibb + \gamma e_1\otimes e_2)$ with a constant rotation $R\in SO(2)$ and scalar valued function $\gamma$. Note that this result holds for any orientation of the slip system $s$.

For the upper bound, we construct recovery sequences, meaning sequences of admissible deformations for $(E_\eps)_\eps$ that are energetically optimal in the limit $\eps\to 0$.
If $s=e_1$, the construction is quite intuitive, one simply compensates for the rigid layers by gliding more in the softer components, namely by a factor $1/\lambda$. Analogue constructions for $s\neq e_1$ are in general not compatible, which makes this case more involved. After suitable approximation and localization, we may focus on affine limit deformations $u$ with gradient $\nabla u=F\in \Mcal_{e_1}\cap \Mcal_s^{\rm qc}$, where $\Mcal^{\rm qc}_s$ denotes the quasiconvex hull of $\Mcal_s$, cf.~\eqref{Me1capNs}.
The observation that admissible sequences which are affine on all layers do not exist due to a lack of appropriate rank-one connections between $\Mcal_s$ and $SO(2)$ (see~Lemma~\ref{lem:rankoneMSO2} and \cite{ChK15PAMM}) motivates to drop the assumption of admissibility at first. Indeed, functions with piecewise constant gradients oscillating between the larger set $\Mcal_s^{\rm qc}$ and $SO(2)$ yield asymptotically optimal energy values. Finally, to make this construction admissible, we glue fine simple laminates with gradients in $\Mcal_s$ into the softer layers, ensuring the preservation of the affine boundary values. This approximating laminate construction, as well as the adaption argument for the boundary, is based on work by Conti and Theil~\cite{CoT05, Con06}, which uses, in particular, convex integration in the sense of M\"uller and \v{S}ver\'ak \cite{MuS99}.

\color{black}
The manuscript is organized as follows. In Section~\ref{sec:rigidity}, we state and prove the asymptotic rigidity result along with a useful corollary. These are the essential ingredients for proving our main result.
We collect some preliminaries on admissible macroscopic deformations in Section~\ref{sec:rankone_lines}, including both necessary conditions and relevant construction tools for laminates that are needed for finding recovery sequences. 
After these preparations, we proceed with the proof of Theorem~\ref{theo:Gammalimit}, which is subdivided into two sections. Section~\ref{sec:Gamma_e1} covers the simpler case $s=e_1$ in a slightly generalized setting, and Section~\ref{sec:proof_sneqe1} gives the detailed proofs for $s\neq e_1$.
Finally, Section~\ref{sec:cell_formula} briefly discusses the relation between the limit functional $E$ and (multi)cell formulas. 
\color{black}

\textit{Notation.}
The standard unit vectors in $\R^2$ are denoted by $e_1, e_2$, and $a^\perp=(-a_2, a_1)$ for $a=(a_1, a_2)\in \R^2$. For the tensor product between vectors $a,b\in \R^2$ we write $a\otimes b=a b^T\in \R^{2\times 2}$. Further, let  $\abs{F}=(FF^T)^{1/2}$ be the Frobenius norm of $F\in \R^{2\times 2}$.
With $\lceil t \rceil$ and $\lfloor t \rfloor$, let us denote the smallest integer not less and largest integer not greater than $t \in \R$, respectively. 
For a set $U \subset \R^2$, the characteristic function $\mathbbm{1}_U$ is given by $\mathbbm{1}_U(x)=1$ for $x\in U$, and $\mathbbm{1}_U(x)=0$ if $x\notin U$. When referring to a domain $\Omega\subset \R^2$, we mean that $\Omega$ is an open, connected, and nonempty set.

Using the standard notation for Lebesgue and Sobolev spaces, we set $L_0^2(\Omega;\R^2) =\{u\in L^2(\Omega;\R^2): \int_{\Omega} u \dd{x} =0\}$, and let $W^{1, 2}_\#(Q;\R^2)$ with a cube $Q\subset \R^2$ stand for the space of $W^{1,2}(Q;\R^2)$-functions with periodic boundary conditions. (Weak) partial derivatives regarding the $i$th variable are denoted by $\partial_i$, and $\nabla u= (\partial_1 u | \partial_2 u) \in \R^{2\times 2}$ 
for a vector field $u:\R^2\to \R^2$. In the two-dimensional setting of this paper, 
the curl operator is defined as follows, $\curl F=\partial_2 Fe_1 - \partial_1 Fe_2$ for $F:\R^2 \to \R^{2\times 2}$.

Notice that we often use generic constants, so that the value of a constant may vary from one line to the other. Moreover, families indexed with $\eps>0$, may refer to any sequence  
$(\eps_j)_j$ with $\eps_j\to 0$ as $j\to 0$.

\section{Asymptotic rigidity of materials with stiff layers}\label{sec:rigidity}

In this section, we examine the qualitative effect of rigid layers on the macroscopic material response of the composite.
The following result provides quite restrictive structural information on volume-preserving effective deformations.
\begin{proposition}[Asymptotic rigidity for layered materials]\label{prop:rigidity}
Let $\Omega\subset \R^2$ be a bounded Lipschitz domain. Suppose that the sequence $(u_\eps)_\eps \subset W^{1,2}(\Omega;\R^2)$ satisfies $u_\eps\weakly u$ in $W^{1,2}(\Omega;\R^2)$ as $\eps\to 0$ for some $u\in W^{1,2}(\Omega;\R^2)$ with $\det \nabla u =1$ a.e.~in $\Omega$, and 
\begin{align*}
\nabla u_\eps \in SO(2)\quad \text{a.e.~in $\Omega\cap \eps Y_{\rm rig}$}
\end{align*}
for all $\eps>0$ with $Y_{\rm rig}$ as defined in~\eqref{def:YrigYsoft}.
Then there exists a matrix $R\in SO(2)$ and $\gamma\in L^2(\Omega)$ such that
\begin{align}\label{representation_nablau}
\nabla u=R(\Ibb+\gamma e_1\otimes e_2).
\end{align}
Furthermore, 
\begin{align}\label{add-on}
\nabla u_\epsilon \mathbbm{1}_{\epsilon\Yrig\cap \Omega} \rightharpoonup \abs{Y_{\rm rig}} R 
\qquad \text{ in $L^2(\Omega; \R^{2\times 2})$}.
\end{align}
\end{proposition}

\begin{remark}\label{rem:rigidity}
a) Considering the model introduced in Section~\ref{sec:introduction}, any weakly converging sequence $(u_\eps)_\eps$ of bounded energy for $(E_\eps)_\eps$ as defined in~\eqref{Eeps_representation2} fulfills the requirements of Proposition~\ref{prop:rigidity}. Indeed, if $E_\eps(u_\eps)<C$ for $\eps>0$, then $u_\eps \in W^{1,2}(\Omega;\R^2)$ and 
\begin{align*}
\nabla u_\eps = R_\eps(\Ibb +\gamma_\eps s\otimes m)
\end{align*}
with $R_\eps\in L^\infty(\Omega; SO(2))$ and $\gamma_\eps \in L^2(\Omega)$ such that $\gamma_\eps=0$ a.e.~in $\Omega\cap\eps \Yrig$, which particularly entails that $\det \nabla u_\eps =1$ a.e.~in $\Omega$.

As a consequence of the weak continuity of the Jacobian determinant~(precisely, $u_\eps\weakly u$ in $W^{1,2}(Q;\R^2)$ implies $\det\nabla u_\eps\weaklystar \det \nabla u$ in the sense of measures, see e.g.~\cite{FLM05} and the references therein), the weak limit  function $u$ satisfies the 
volume constraint $\det \nabla u =1$ a.e.~in $\Omega$. 

In fact, \eqref{representation_nablau} provides a necessary condition for the class of admissible deformations in the effective limit model. It indicates that, macroscopically, (up to a global rotation) only horizontal shear can be achieved.

b) Notice that due to the gradient structure of $\nabla u$ in~\eqref{representation_nablau}, 
the function $\gamma$ is independent of $x_1$ in the sense that its distributional derivative $\partial_1 \gamma$ vanishes. This follows immediately from $0=\curl \nabla u=- \partial_1\gamma Re_1$. 
\end{remark}
The outline of the proof of Proposition~\ref{prop:rigidity} is as follows. 
First, we conclude from the well-known rigidity result in Lemma~\ref{lem:rigidity}, applied
to the connected components of $\Omega \cap \eps \Yrig$, that each stiff layer can only be rotated as a whole. 
The resulting rotation matrices are then used to construct a sequence of one-dimensional piecewise constant auxiliary functions for which we  establish compactness and from which we obtain structural information on $\nabla u$. 
More precisely, as a consequence of the explicit estimate in Lemma~\ref{lem:1d estimate}, the rotations of neighboring stiff layers are close for small $\eps$, and the auxiliary sequence has bounded variation.
By Helly's selection principle one can extract a pointwise converging subsequence whose limit function lies in $SO(2)$ a.e.~in $\Omega$, since lengths are preserved in this limit passage. Along with $\det \nabla u=1$, this observation translates into the representation $\nabla u = R(\Ibb +\gamma e_1\otimes e_1)$ with $R\in SO(2)$ a.e.~in $\Omega$. Finally, to prove that $R$ is constant, we exploit essentially the gradient structure of $\nabla u$.

Before giving the detailed arguments, let us briefly state one of the key tools, 
which, in its classical version, 
is also known as Liouville's theorem. The first proof in the context of Sobolev maps goes back to Reshetnyak~\cite{Res67}, for a quantitative generalization of the result we refer to~\cite[Theorem~3.1]{FJM02}.

\begin{lemma}[Rigidity for Sobolev functions]\label{lem:rigidity}
Let $\Omega\subset \R^2$ be a bounded Lipschitz domain and $u\in W^{1,2}(\Omega;\R^2)$ with $\nabla u(x) \in SO(2)$ for a.e.~$x\in \Omega$. Then $u$ is harmonic and there is a constant rotation $R\in 
SO(2)$ such that $\nabla u(x)=R$ for all $x\in \Omega$. In particular, $u(x) = Rx+b$ for some $b\in \R^2$. 
\end{lemma}

An explicit estimate of the distance between rotations of neighboring stiff layers is given in the following lemma.

\begin{lemma} \label{lem:1d estimate}
Let $P = (0,L)\times (0, H)$ with $L, H >0$. For $i=1,2$, let $w_i:P\to \R^2$ be the affine functions defined by $w_i(x)=R_ix+b_i$ with $R_i \in SO(2)$ and $b_i \in \R^2$. 

If  $u \in W^{1,2}(P;\R^2)$ is such that
\begin{align*}
u= w_1\ \text{on $\partial P \cap \{x_2 = 0\}$}\quad \text{and}\quad u = w_2 \text{ on $\partial P \cap \{x_2 = H\}$}
\end{align*}
in the sense of traces, then
\begin{align}
\int_P |\nabla ue_2|^2 \dd{x} \geq \frac{L^3}{24H} \abs{R_1 - R_2}^2. \label{lowerboundI}
\end{align}
\end{lemma}

\begin{proof}
We observe that for given $a,b\in \R^2$ the $1$-d minimization problem 
\begin{align}\label{1dminimization}
\inf\bigg\{ \int_0^{H} |v'(t)|^2 \dd{t}: v \in W^{1,2}(0, H;\R^2), v(0) = a, v(H) = b\bigg\}
\end{align}
has a unique solution. Indeed, by Jensen's inequality, the minimizer $\bar{v}$ of~\eqref{1dminimization} is given by linear interpolation as $\bar{v}(t) = \tfrac{1}{H}(b-a)t+ a$ for $t \in (0, H)$.
Since $u\in W^{1,2}(P;\R^2)$, one has  that $u(x_1, \frarg)\in W^{1,2}(0,H;\R^2)\subset AC([0,H];\R^2)$  with $u(x_1, 0) = w_1(x_1, 0)$ and $u(x_1, H)=w_2(x_1, H)$ for a.e.~$x_1\in (0,L)$.
By setting
\begin{align*}
\bar{u}(x)= \frac{x_2}{H}(w_2(x_1, H)-w_1(x_1, 0))+ w_1(x_1, 0), \qquad x\in P,
\end{align*}
we therefore obtain 
\begin{align*}
\int_P \abs{\nabla u e_2}^2 \dd{x} & =\int_P \abs{\partial_2 u(x)}^2 \dd{x} \geq \int_0^L \int_0^H \abs{\partial_2 \bar{u}(x_1, x_2)} ^2\dd{x_2}\dd{x_1} \\ &=\frac{1}{H} \int_0^L \abs{w_2(x_1, H) -w_1(x_1, 0)}^2\dd{x_1} \\ & = 
\frac{1}{H} \int_0^L \abs{x_1 (R_2-R_1)e_1 + HR_2e_2  + (b_2-b_1)}^2\dd{x_1}.
\end{align*}
Minimizing this expression with respect to $b_1$ and $b_2$ gives~\eqref{lowerboundI}. 
\end{proof}

\begin{proof}[Proof of Proposition~\ref{prop:rigidity}]
To characterize the limit function $u$, 
we will show that the statement holds locally, i.e.~on any open cube $\ucal\subset\Omega$ with sides parallel to the coordinate axes. Precisely, there exists a rotation $R_{\ucal}\in SO(2)$ and $\gamma_{\ucal}\in L^2(\Omega)$ such that $\nabla u|_{\ucal}=R_{\ucal}(\Ibb +\gamma_{\ucal} e_1\otimes e_2)$. To deduce~\eqref{representation_nablau}, it suffices to exhaust $\Omega$ with overlapping cubes $Q$. This way one finds that all $R_Q$ coincide, leading to a global rotation $R\in SO(2)$.

Without loss of generality, let us assume in the following that $Q=(0, l)^2$ with $l>0$. 
To describe the layered structure of the material, we introduce the notation 
\begin{align*}
P_\eps^i= (\R\times \eps[i, i+1)) \cap \Ucal, \qquad  i \in \Z,\, \eps>0,
\end{align*} 
for the horizontal strips in a larger open cube $\Ucal\subset\Omega$ that compactly contains $Q$.
The index set $I_\eps = \{i \in \Z : \abs{P_\eps^i} = \eps \sqrt{\abs{\Ucal}}\}$ selects those strips of thickness $\eps$ that are fully contained in $\Ucal$. Then, by taking $\eps$ sufficiently small 
one has that $Q \subset \bigcup_{i \in I_\eps} P^i_\eps$.

We subdivide the remaining proof in six steps.

\textit{Step~1: Classical rigidity and approximation by piecewise affine functions.}
Applying Lemma~\ref{lem:rigidity} to each strip $P^i_\epsilon$ with $i\in I_\eps$ 
yields the existence of rotation matrices $R_\eps^i\in SO(2)$ and translation vectors $b_\eps^i\in \R^2$ such that
\begin{align*}
u_\eps(x)=R_\eps^i x+b_\eps^i, \qquad x\in P_\eps^i\cap \eps Y_{\rm rig}.
\end{align*}
Let the sequences $(\sigma_\eps)_\eps\subset L^{\infty}(\ucal;\R^2)$ 
and $(b_\eps)_\eps\subset L^\infty(\ucal;\R^2)$ be defined by
\begin{align*}
\sigma_\eps(x) =\sum_{i\in I_\eps}(R_\eps^ix)\mathbbm{1}_{P_\eps^i}(x)\qquad \text{and}\qquad b_\eps(x)=\sum_{i\in I_\eps} b_\eps^i\mathbbm{1}_{P_\eps^i}(x), \qquad x\in \ucal,
\end{align*}
and let $w_\eps=\sigma_\eps+b_\eps$. 
Next we show that
\begin{align}\label{convergence}
\lim_{\eps\to 0}\norm{u_\eps-w_\eps}_{L^2(\ucal;\R^2)}=0.
\end{align}
For each $i\in I_\eps$, we apply a 1-d version of the Poincar\'{e} inequality  
to derive that
\begin{align*}
\int_{P_\eps^i\cap \ucal} \abs{u_\eps-w_\eps}^2\dd{x} &=
\int_{0}^{l}\int_{\eps i}^{\eps (i+1)} \abs{u_\eps-w_\eps}^2\dd{x_2}\dd{x_1}
\leq \int_{0}^{l} c\eps^2 \int_{\eps i}^{\eps(i+1)} \abs{\partial_2 u_\eps - \partial_2 \sigma_\eps}^2 \dd{x_2}\dd{x_1} \\
& =c \eps^2 \int_{P_\eps^i\cap \ucal} \abs{(\nabla u_\eps-R_\eps^i)e_2}^2\dd{x} 
 \leq c\eps^2\bigl(\norm{\nabla u_\eps}^2_{L^2(P_\eps^i;\R^{2\times 2})}+ \abs{P_\eps^i}\bigr)
\end{align*}
with constants $c>0$ independent of $\eps$. Summing over all $i\in I_\eps$ gives
\begin{align*}
\norm{u_\eps - w_\eps}_{L^2(Q;\R^2)}^2 \leq c\eps^2 \bigl( \norm{u_\eps}_{W^{1,2}(\Omega;\R^{2})}+ \abs{\Omega}\bigr) \leq c\eps^2,
\end{align*}
and thus,~\eqref{convergence}.

We point out that $(\sigma_\eps)_\eps$ and $(b_\eps)_\eps$ are uniformly bounded in $L^\infty(Q;\R^2)$ and $L^2(Q;\R^2)$, respectively. The latter follows together with~\eqref{convergence} and the uniform boundedness of $(u_\eps)_\eps$ in $L^2(\Omega;\R^{2})$.
Consequently, there are subsequences of $(\sigma_\eps)_\eps$ and $(b_\eps)_\eps$ (not relabeled) and functions $\sigma, b\in L^2(Q;\R^2)$ such that
\begin{align}\label{sigmabetaconvergence}
\sigma_{\eps}\weaklystar \sigma\quad \text{in $L^\infty(Q;\R^2)$}\qquad \text{and}\qquad
b_{\eps}\weakly b\quad \text{in $L^2(Q;\R^2)$.}
\end{align}
Hence, $w_\eps\weakly \sigma + b$ in $L^2(Q;\R^2)$, so that, in view of ~\eqref{convergence} and the uniqueness of weak limits, 
\begin{align}\label{characterization_limitu}
u=\sigma + b.
\end{align}
Notice that $\partial_1 b=0$ due to the fact that the functions $b_{\eps}$ are independent of $x_1$ considering the definition of the strips $P_\eps^i$.

\textit{Step~2: Compactness by Helly's selection principle.}
It follows from Lemma~\ref{lem:1d estimate}, applied to the suitably shifted softer layers, that
\begin{align}\label{443}
\int_{\Ucal}\abs{\nabla u_\eps e_2}^2 \dd{x}\geq \sum_{i\in I_\eps, i>i_\eps} 
\int_{P_\eps^i\cap \eps \Ysoft} \abs{\nabla u_\eps e_2}^2 \dd{x} \geq \frac{|\Ucal|^{3/2}}{24 \eps \lambda}\sum_{i\in I_\eps, i>i_\eps} \abs{R_\eps^i-R^{i-1}_\eps}^2
\end{align}
with $i_\eps$ the smallest integer in $I_\eps$.
Since $(u_\eps)_\eps$ is uniformly bounded in $W^{1,2}(\Omega;\R^2)$, one infers that
\begin{align}\label{est_rigid}
\Bigl(\sum_{i\in I_\eps, i>i_\eps} \abs{R_\eps^i - R_\eps^{i-1}} \Bigr)^2\leq \# I_\eps \sum_{i\in I_\eps, i>i_\eps} \abs{R_\eps^i -R_\eps^{i-1}}^2 \leq C
\end{align}
for all $\eps$ with a constant $C>0$. Besides~\eqref{443}, we use for the last estimate that the cardinality of $I_\eps$ satisfies $\# I_\eps\leq \eps^{-1}\abs{\Ucal}^{1/2}$.

Consider now the piecewise constant function of one real variable $\Sigma_\eps\in L^\infty(0,l;SO(2))$
given by
\begin{align*}
\Sigma_\eps(t)=\sum_{i\in I_\eps} R_\eps^i \mathbbm{1}_{\eps[i, i+1)}(t), \qquad t\in (0,l),
\end{align*}
with the rotation matrices $R_\eps^i$ of Step~1. In view of~\eqref{est_rigid} the sequence $(\Sigma_\eps)_\eps$ has uniformly bounded variation. Then, by Helly's selection principle we can find a suitable (not relabeled) subsequence of 
$(\Sigma_{\eps})_\eps$ and a function $\Sigma:(0,l)\to \R^{2\times 2}$ of bounded variation such that
\begin{align}\label{pointwiseconvergence}
\Sigma_{\eps}(t)\to \Sigma(t)\quad \text{for all $t\in (0,l).$} 
\end{align}
Since $\det \Sigma_\eps(t)=1$ and $\abs{\Sigma_\eps(t)e_1}=1$ for all $t\in (0,l)$ and all $\eps>0$, it follows by continuity that 
$\Sigma(t)\in SO(2)$ for all $t\in (0, l)$.

\textit{Step~3: Improved regularity for $\Sigma$.}
Let $\Pi_\epsilon: (0,l) \rightarrow \R^{2\times2}$ be the linear interpolant of $\Sigma_\eps$ between the 
points $\eps(i+\frac{1}{2})\in (0,l)$ with $i\in I_\eps$. 
In the intervals close to the endpoints not covered by this definition, we take $\Pi_\eps$ to be constant.

As a continuous, piecewise affine function $\Pi_\epsilon$ is almost everywhere differentiable, and together with~\eqref{est_rigid} or~\eqref{443} it holds that
\begin{align*}
\int_0^l \abs{\Pi_\epsilon'}^2 \dd{t} \leq \sum_{i\in I_\epsilon, i>i_\eps} \eps \frac{\abs{R_\epsilon^i - R_\epsilon^{i-1}}^2}{\epsilon^2}  \leq C,
\end{align*}
and
\begin{align}\label{est17}
\int_ 0^l \abs{\Pi_\epsilon - \Sigma_\epsilon}^2 \dd{t} \leq \sum_{i \in I_\epsilon, i>i_\eps}\eps  \abs{R_\epsilon^i - R_\epsilon^{i-1}}^2 \leq C\eps^2.
\end{align}
In particular,  $(\Pi_\epsilon)_\eps$  is uniformly bounded in $W^{1,2}(0,l; \R^{2\times 2})$, and 
therefore $(\Pi_{\eps})_\eps$ 
admits a weakly converging subsequence with limit $\Pi\in W^{1,2}(0,l;\R^{2\times 2})$.
From~\eqref{pointwiseconvergence}, \eqref{est17}, and the uniqueness of the limit we infer that
 \begin{align}\label{regularity}
 \Sigma=\Pi \in W^{1,2}(0,l;SO(2)).
 \end{align}
By constant extension of $\Sigma$ in $x_1$-direction we define a map $R$ on $Q$, 
precisely we set $R(x)=\Sigma(x_2)$ for $x\in Q$. Then, $R\in W^{1,2}(Q;SO(2))$.

\textit{Step~4: Establishing $\nabla u\in \Mcal_{e_1}$ pointwise.}
The estimate
\begin{align*}
\abs{\sigma_\eps(x)-\Sigma(x_2) x} \leq\absB{\sum_{i\in I_{\eps}} \big(R_\eps^i-\Sigma(x_2)\big)\mathbbm{1}_{P_\eps^i}(x)}\abs{x}\leq \sqrt{2}l\absb{\Sigma_\eps(x_2)-\Sigma(x_2)}, \qquad x\in Q,
\end{align*}
along with~\eqref{pointwiseconvergence} 
and 
the first part of~\eqref{sigmabetaconvergence} leads to 
\begin{align}\label{eq99}
\sigma(x) = \Sigma(x_2) x = R(x)x\qquad \text{ for a.e.~$x\in Q$.}
\end{align} 
From~\eqref{characterization_limitu} and the independence of $b$ of $x_1$ we then conclude that
\begin{align}\label{eq1}
\nabla u e_1= R e_1.
\end{align}
This shows in particular that $R$, $\Sigma$, $\sigma$, and $b$ are independent of the choice of subsequences in~\eqref{pointwiseconvergence} and~\eqref{sigmabetaconvergence}. 
Moreover, by~\eqref{regularity} and~\eqref{eq99} it is immediate to see that $\sigma, b\in W^{1,2}(Q;\R^2)$.

Since $R \in SO(2)$ pointwise by Step~3, one has that $\abs{\nabla u e_1}=1$ a.e.~in $Q$. 
In conjunction with $\det \nabla u =1$ a.e.~in $Q$, we conclude that $\nabla u\in \Mcal_{e_1}$ a.e.~in $Q$. 
In view of~\eqref{eq1}, there exists a function $\gamma \in L^2(Q)$ such that
\begin{align}\label{first_representation}
\nabla u = R\big(\Ibb + \gamma e_1 \otimes e_2\big).
\end{align}

\textit{Step~5: Proving $R$ constant.}
 Using~\eqref{characterization_limitu} and~\eqref{eq99}, we compute that
 \begin{align*}
 \nabla ue_2(x)= \partial_2 R(x)x + R(x)e_2 + \partial_2 b(x)
 \end{align*}
for a.e.~$x\in Q$. Then, along with~\eqref{first_representation} and the independence of $R$ of $x_1$, it follows for the distributional derivative of $\gamma$ that
\begin{align}\label{partial1gamma}
\partial_1 \gamma = \partial_1(\nabla ue_2\cdot Re_1) = \partial_2 Re_1\cdot Re_1 = \partial_2 \abs{Re_1}^2=0.
\end{align}
As $\curl \nabla u=0$ in $Q$ in the sense of distributions, the representation~\eqref{first_representation} entails
\begin{align*}
0 = \int_Q Re_1 \partial_2 \varphi - Re_2\partial_1\varphi- \gamma Re_1 \partial_1\varphi \dd{x} =\int_Q Re_1 \partial_2\varphi \dd{x}
\end{align*}
for all $\varphi \in C_c^\infty(Q)$,
where we have used $\partial_1 R=0$ and~\eqref{partial1gamma}. This shows $\partial_2 Re_1=0$, which implies that $R$ is a constant rotation.

\textit{Step~6: Proof of~\eqref{add-on}.}
Accounting for~\eqref{pointwiseconvergence} and $\norm{\Sigma_\eps}_{L^\infty(Q;\R^{2\times 2})}=2$, one finds that $\Sigma_\epsilon\to\Sigma$ in $L^2(0,l;\R^{2\times 2})$. Together with 
$\mathbbm{1}_{\eps \Yrig}\weaklystar \abs{Y_{\rm rig}}=(1-\lambda)$ in $L^\infty(Q)$, a weak-strong convergence argument leads to
\begin{align*}
\nabla u_\epsilon \mathbbm{1}_{\epsilon \Yrig} 
  \rightharpoonup (1-\lambda) R\quad \text{in $L^2(Q;\R^{2\times 2})$},
\end{align*}
considering that $\nabla u_\eps(x) = \Sigma_\eps(x_2)$ for a.e.~$x\in \eps\Yrig\cap Q$.
To see that the statement holds in $L^2(\Omega;\R^{2\times 2})$ as well, we argue again by exhaustion of $\Omega$ with cubes $Q$.
\end{proof}

As discussed in Remark~\ref{rem:rigidity}, Proposition~\ref{prop:rigidity} imposes structural restrictions on the limits of bounded energy sequences for $(E_\eps)_\eps$. As a consequence, we obtain an asymptotic lower bound energy estimate, which constitutes a first step toward the proof of the liminf-inequality~\eqref{liminf} for the $\Gamma$-convergence result in Theorem~\ref{theo:Gammalimit}.

\begin{corollary}\label{cor: estimate lower bound}
Let $(u_\epsilon)_\eps \subset W^{1,2}(\Omega; \R^2)$ be such that $E_\eps(u_\eps)\leq C$ for all $\eps>0$ and $u_\eps \weakly u$ in $W^{1,2}(\Omega;\R^2)$ for some $u\in W^{1,2}(\Omega;\R^2)$ with gradient of the form~\eqref{representation_nablau}. 
If, in addition, $u$ is (finitely) piecewise affine, 
then 
\begin{align*}
 \liminf_{\epsilon\rightarrow0} \int_\Omega |\nabla u_\epsilon m|^2 - 1 \dd x\geq  \int_\Omega \frac{1}{\lambda} |\nabla um - (1-\lambda)Rm|^2 - \lambda\dd{x}.
\end{align*}
\end{corollary}

\begin{proof} One may assume in the following that $u$ is affine, otherwise the same arguments can be applied to each affine piece of $u$.

Since $\nabla u_\epsilon \rightharpoonup \nabla u$ in $L^2(\Omega;\R^{2\times 2})$, and by~\eqref{add-on}, $\nabla u_\epsilon\mathbbm{1}_{\epsilon\Yrig} \rightharpoonup (1-\lambda)R$ in $L^2(\Omega; \R^{2\times2})$, it follows that
\begin{align}\label{eq24}
\nabla u_\epsilon\mathbbm{1}_{\epsilon \Ysoft} \rightharpoonup \nabla u - (1-\lambda)R 
\quad \text{in~} L^2(\Omega; \R^{2\times 2}),
\end{align}
and thus, in particular in $L^1(\Omega; \R^{2\times 2})$. 
From $\abs{\nabla u_\eps m}=1$ a.e.~in $\eps\Yrig\cap \Omega$, H\"older's inequality, and~\eqref{eq24} in conjunction with the weak lower semicontinuity of the $L^1$-norm we infer that 
\begin{align*}
\liminf_{\eps\to 0} \int_{\Omega} \abs{\nabla u_\eps m}^2-1\dd{x} & =\liminf_{\eps\to 0} \int_\Omega \abs{\nabla u_\eps m\, \mathbbm{1}_{\eps\Ysoft}}^2\dd{x}-\abs{\Omega\cap \eps\Ysoft}\\
& \geq \liminf_{\eps\to 0} \abs{\Omega\cap \eps\Ysoft}^{-1} \Bigl(\int_{\Omega} \abs{\nabla u_\eps m\,\mathbbm{1}_{\eps\Ysoft}} \dd{x}\Bigr)^2 -\abs{\Omega\cap\eps\Ysoft}\\ 
& \geq\frac{\abs{\Omega}}{\lambda}\abs{\nabla um-(1-\lambda)Rm}^2-\lambda\abs{\Omega}. 
\end{align*}
Notice that in the last step we also used that $\abs{\Omega\cap\eps\Ysoft} = \int_{\Omega}\mathbbm{1}_{\eps\Ysoft}\dd{x}\to \lambda\abs{\Omega}$ as $\eps\to 0$, as well as the assumption that $\nabla u$ is constant.
\end{proof}
\color{black}


\section{Discussion of admissible deformations}\label{sec:rankone_lines}
In preparation for the proof of Theorem~\ref{theo:Gammalimit}, we exploit the specific form of the functionals $E_\eps$ to identify further properties of the weak limits of bounded energy sequences. Moreover, we provide the basis for the laminate constructions that are the key to obtaining suitable recovery sequences.

\subsection{Necessary conditions for admissible macroscopic deformations}
Let $(u_\eps)_\eps\subset$ \break $W^{1,2}(\Omega;\R^2)$ satisfy $\nabla u_\eps \in \Mcal_s$ a.e.~in $\Omega$, 
and suppose that $u_\eps\weakly u$ in $W^{1,2}(\Omega;\R^2)$ for some $u\in W^{1,2}(\Omega;\R^2)$. As the convex set $\{F\in L^2(\Omega;\R^{2\times 2}): \abs{Fs}\leq 1\text{ a.e.~in $\Omega$}\}$ is weakly closed in $L^2(\Omega;\R^{2\times 2})$ and $\det \nabla u_\eps \weaklystar \det \nabla u$ in the sense of measures (cf.~Remark~\ref{rem:rigidity}\,a)), we know that
\begin{align}\label{Msqc}
\nabla u\in \Ncal_s \quad \text{a.e.~in $\Omega$},
\end{align}
where $\Ncal_s=\{F\in \R^{2\times 2}: \det F=1, \abs{Fs}\leq 1\}$.
According to~\cite{CoT05} (see also \cite{CDK11}), the set $\Ncal_s$ is exactly the quasiconvex hull $\Mcal_s^{\rm qc}$ of $\Mcal_s$.
With $S=(s|m)=(s|s^\perp)\in SO(2)$, another alternative representation of $\Ncal_s$ is 
\begin{align*}
\Ncal_s= \big\{F\in \R^{2\times 2}: F= R\big(S (\alpha e_1\otimes e_1 + \tfrac{1}{\alpha}e_2\otimes e_2) S^T + \gamma s\otimes m\big), R\in SO(2),\alpha\in (0,1], \gamma\in \R\big\}.\end{align*} 
 One also has that
\begin{align}\label{representationMsNs}
\Mcal_s=\{F\in \R^{2\times 2}: F=GS^T, G\in \Mcal_{e_1}\} \quad\text{and}\quad \Ncal_s=\{F\in \R^{2\times 2}: F= GS^T, G\in \Ncal_{e_1}\}. 
\end{align}

If we assume in addition that $(u_\eps)_\eps$ is a sequence of bounded energy, precisely, $E_\eps(u_\eps)<C$ for all $\eps$, then 
$\nabla u \in \Mcal_{e_1}$ pointwise almost everywhere in $\Omega$
by the rigidity result in Proposition~\ref{prop:rigidity} and Remark~\ref{rem:rigidity}\,a). Thus, together with~\eqref{Msqc},
\begin{align}\label{Me1capNs}
\nabla u \in \Mcal_{e_1} \cap \Ncal_s\quad \text{a.e.~in $\Omega$.}
\end{align}
For $s=e_1$ the restriction in~\eqref{Me1capNs} is equivalent to $\nabla u\in \Mcal_{e_1}$ a.e., while for $s\neq e_1$ a straightforward computation shows that 
\begin{align}\label{char_Me1capNs}
\Mcal_{e_1}\cap \Ncal_s= & \{F\in \R^{2\times 2}: F=R(\Ibb+\gamma e_1\otimes e_2), R\in SO(2), \gamma \in K_{s,1} \},
\end{align}
with $K_{s, 1}$ as defined in~\eqref{def_Ks}.

In the case $s\neq e_1$, condition~\eqref{Me1capNs} can be refined even further by exploiting the presence of the rigid layers with their asymptotic volume fraction $\abs{Y_{\rm rig}} = 1-\lambda$.
Indeed, from Proposition~\ref{prop:rigidity} 
we infer that there are $R\in SO(2)$ and $\gamma\in L^2(\Omega)$ such that 
\begin{align*}
(\nabla u_\eps s)\mathbbm{1}_{\eps\Ysoft} = \nabla u_\eps s- (\nabla u_\eps s)\mathbbm{1}_{\eps\Yrig} \weakly  R(\lambda \Ibb  + \gamma e_1\otimes e_2)s \quad \text{in $L^2(\Omega;\R^{2})$,}
\end{align*}
and therefore also in $L^1(\Omega;\R^2)$.
On the other hand, $\abs{(\nabla u_\eps s) \mathbbm{1}_{\eps\Ysoft}} =\mathbbm{1}_{\eps\Ysoft}\weaklystar \lambda$ in $L^\infty(\Omega)$. By the weak lower semicontinuity of the $L^1$-norm, we obtain for any open ball $B\subset \Omega$ that
\begin{align*}
\int_B \abs{R(\lambda \Ibb + \gamma e_1\otimes e_2)s}\dd{x} \leq \lim_{\eps\to 0} \int_B \abs{(\nabla u_\eps s)\mathbbm{1}_{\eps\Ysoft}} \dd{x}= \abs{B}\lambda,
\end{align*}
and consequently,
\begin{align*}
\dashint_B \abs{\lambda s + \gamma s_2 e_1}\dd{x} \leq \lambda.
\end{align*}
Applying Lebesgue's differentiation theorem entails the pointwise estimate $\abs{\lambda s+ \gamma s_2 e_1}\leq \lambda$ a.e.~in $\Omega$, which is equivalent to
\begin{align*}
\gamma \in K_{s, \lambda} \quad \text{a.e.~in $\Omega$,}
\end{align*}
cf.~\eqref{def_Ks} for the definition of $K_{s, \lambda}$.

\subsection{Tools for the construction of admissible deformations}\label{subsec:construction_tools}

We start by characterizing all rank-one connections in $\Mcal_s$, cf.~also~\cite{ChK15PAMM}.
\begin{lemma}[Rank-one connections in \boldmath{$\Mcal_s$}]\label{lem:rankoneMSO2}
Let $F, G \in \Mcal_s$ such that $F= R (\mathbb{I} + \gamma s\otimes m)$ and $G=Q(\Ibb +\zeta s\otimes m)$ with $R, Q\in SO(2)$ and $\gamma, \zeta\in \mathbb{R}$. 
Then $F$ and $G$ are rank-one connected, i.e.~$\rank (F-G)$ =1, if and only if one of the following relations holds:

 $i)$ $R = Q$ and $\gamma \neq \zeta$,\\
 or 
 
 $ii)$ $R \neq Q$ and $\gamma-\zeta = 2 \tan(\theta/2)$, where $\theta\in (-\pi, \pi)$ denotes the rotation angle of $Q^TR$, meaning that $Q^TRe_1=\cos \theta e_1+ \sin \theta e_2$.
 
In particular, in case of $i)$, $F-G=(\gamma -\zeta )Rs\otimes m$, while for $ii)$ one has 
\begin{align}\label{eq20}
F-G=\frac{\gamma-\zeta}{4+(\gamma-\zeta)^2}Q((\zeta-\gamma)s+2m) \otimes (2s+(\gamma +\zeta)m).
\end{align} 
\end{lemma}

\begin{proof}
Considering~\eqref{representationMsNs}, it is enough to prove the statement for $s=e_1$. Moreover, we may assume without loss of generality that $Q=\Ibb$. 
It follows from $\det(F-G) = 0$ that 
\begin{align*}
Re_1\cdot (2e_1+ (\gamma-\zeta)e_2)=2.
\end{align*}
Thus, either $Re_1=e_1$, i.e.~$R=\Ibb$, or 
\begin{align}\label{eq18}
Re_1= \frac{4-(\gamma-\zeta)^2}{4+(\gamma-\zeta)^2}e_1 +  \frac{4(\gamma-\zeta)}{4+(\gamma-\zeta)^2}e_2.
\end{align} 
In view of the definition of $\theta$ and some basic identities for trigonometric functions,~\eqref{eq18} is equivalent to $\gamma-\zeta=2\tan(\theta/2)$. The representation of $F-G$ in~\eqref{eq20} is then straightforward to compute. 
\end{proof}

\begin{remark}\label{rem:rank1_connections}
As this paper is concerned with materials built from horizontal layers, we are especially interested in rank-one connections with normal $e_2$, i.e.~$F, G\in \Mcal_s$ as in Lemma~\ref{lem:rankoneMSO2} with $\rank(F-G)=1$ satisfying $F-G=a\otimes e_2$ for some $a\in \R^2\setminus \{0\}$. 
If $s=e_1$, this implies $R=Q$, but there are no restrictions on $\gamma$ and $\zeta$ other than $\gamma\neq \zeta$.
For $s\neq e_1$ one needs that $\gamma+\zeta=2\frac{s_1}{s_2}$. Hence, for given $\gamma\in \R$ also the rotation matrix $Q^TR$ is uniquely determined in this case.
\end{remark}

It was first proven in~\cite{CoT05} that $\Ncal_s=\Mcal_s^{\rm qc}$ coincides with the rank-one convex hull $\Mcal_s^{\rm rc}$, which in particular, means that every $N\in \Ncal_s$ can be expressed as a convex combination of rank-one connected matrices in $\Mcal_s$. A specific type of rank-one directions, which turns out optimal for the relaxation of $W_{\rm soft}$, was discussed by Conti in~\cite{Con06}, see also~\cite{CDK11}. Here we give a different argumentation based on Lemma~\ref{lem:rankoneMSO2}.
 
\begin{lemma}\label{lem:construction_rankoneconnections}
For a given $N\in \Ncal_s\setminus\Mcal_s$ there are $F, G\in \Mcal_s$ as in Lemma~\ref{lem:rankoneMSO2} with $\rank(F-G)=1$ and $\mu\in (0,1)$ such that 
\begin{align}\label{eq19}
N=\mu F+(1-\mu)G \qquad\text{and}\qquad \abs{Nm}=\abs{Fm}=\abs{Gm}. 
\end{align}
\end{lemma}
\begin{proof} Let $N\in \Ncal_s\setminus \Mcal_s$.
We determine $\mu\in (0,1)$ as well as $Q, R\in SO(2)$ and $\zeta, \gamma\in \R$ such that the desired properties are satisfied for $F=R(\Ibb + \gamma s\otimes m)$ and $G=Q(\Ibb +\zeta s\otimes m)$.

Since the second condition in~\eqref{eq19} is equivalent to $\gamma$ and $\zeta$ satisfying $\abs{\gamma}^2=\abs{\zeta}^2=\abs{Nm}^2-1$,
we may choose
\begin{align}\label{eq21}
\gamma=  \sqrt{\abs{Nm}^2-1}\qquad\text{and}\qquad \zeta=-\gamma.
\end{align}
Notice that $\abs{Nm}>1$, as $1=\det N\leq \abs{Ns}\abs{Nm}$ and $\abs{Ns}<1$ by assumption.
Then, by Lemma~\ref{lem:rankoneMSO2}, $F$ and $G$ are rank-one connected, if $Q^TR$ corresponds to the rotation with angle $\theta=2\arctan \gamma$. We use this relation to define $R$ for given $Q$ to be determined in the next step.

For the first part of~\eqref{eq19}, 
it is necessary that
\begin{align}\label{eq22}
Ns= Gs+\mu(F-G)s=Q\Bigl(s+\frac{2\gamma\mu}{1+\gamma^2}(m-\gamma s)\Bigr),
\end{align} 
where we have used~\eqref{eq20} along with~\eqref{eq21}.
Taking squared norms in the above equation imposes a constraint on $\mu$ in the form of a quadratic equation, which has 
two solutions $\mu_1\in(0,1/2)$ and  $\mu_2\in (1/2,1)$ with $\mu_1+\mu_2=1$.  
Depending on which of these values is selected for $\mu$, we adjust the rotation $Q$ so that~\eqref{eq22} holds. 
It follows from~\eqref{eq22} and~\eqref{eq21} that
\begin{align*}
 Gm \in \Acal_N=\{a\in \R^2: (Ns)^\perp\cdot a =1, \abs{a}=\abs{Nm}\}.
\end{align*} 
As $\abs{Nm}^{-1}\leq \abs{Ns}$, the set $\Acal_N$ contains exactly two elements, one of which being $Nm$. 
Finally, we take $\mu\in (0,1)$ with corresponding $Q$ such that $Gm=Nm$, which finishes the proof.
\color{black}

Let us remark that choosing $\gamma=-\sqrt{\abs{Nm}^2-1}$ in~\eqref{eq21} essentially comes up to switching $F$ and $G$.
\end{proof}

In the case of a non-horizontal slip direction, optimal constructions of admissible deformations cannot be achieved based on rank-one connections in $\Mcal_s$. Instead, we employ simple laminates with gradients in $SO(2)$ and $\Ncal_s$ (and normal $e_2$). The following one-to-one correspondence between $\gamma\in K_{s, \lambda}$ and $R\in SO(2)$, and $N\in \Ncal_s$ with $\abs{Ne_1}=1$ is helpful for the explicit constructions.

\begin{lemma} \label{lem: specific laminates in N}
Let $\lambda\in (0,1)$ and $s\in \R^2$ with $\abs{s}=1$ and $s\neq e_1$ be given.

i) For every $\gamma\in K_{s, \lambda}$ and $R\in SO(2)$ there exists $N\in  \Mcal_{e_1}\cap \Ncal_s$ such that $Ne_1=Re_1$ and
\begin{align}\label{equality_construction}
\lambda N + (1-\lambda)R=R(\Ibb+\gamma e_1\otimes e_2).
\end{align}

ii) Let $N \in \Ncal_s$ and $R\in SO(2)$ with $Re_1=Ne_1$. Then there exists $\gamma\in K_{s, \lambda}$ such that
\eqref{equality_construction} is satisfied.
\end{lemma}

\begin{proof}
For the proof of $i)$ we set 
\begin{align}\label{Ngamma}
N=R(\Ibb+\tfrac{\gamma}{\lambda} e_1\otimes e_2).
\end{align} 
It is immediate to check that $Ne_1=Re_1$,~\eqref{equality_construction} is fulfilled, $N\in \Mcal_{e_1}$, and $\abs{Ns} = \abs{s+ \tfrac{\gamma}{\lambda} s_2 e_1}\leq 1$ 
in view of $\gamma\in K_{s, \lambda}$.

As regards $ii)$, choosing $\gamma=\lambda Ne_1\cdot Ne_2$ gives the desired element in $K_{s, \lambda}$. Indeed, $R(\Ibb+\gamma e_1\otimes e_2)e_2= Re_2 +\lambda(Ne_1 \cdot Ne_2)Ne_1= Re_2 + \lambda Ne_2 - \lambda ((Ne_1)^\perp\cdot Ne_2) (Ne_1)^\perp=\lambda Ne_2 + (1-\lambda)Re_2$ in view of $1=\det N=(Ne_1)^\perp\cdot Ne_2$, which along with $Ne_1=Re_1$ proves~\eqref{equality_construction}. 
Since~\eqref{equality_construction} implies that $N$ is of the form \eqref{Ngamma}, it follows from a direct computation that $\gamma\in K_{s, \lambda}$. 
\color{black}
\end{proof}

The following two theorems are taken from Conti \& Theil~\cite{CoT05} and M\"uller \& \v{S}ver\'{a}k~\cite{MuS99}, respectively. In combination, they allow us modify a simple laminate with gradients in $\Mcal_s$ in a small part of the domain in such a way that the resulting Lipschitz function takes affine boundary values in $\Ncal_s$, while preserving the constraint that gradients lie pointwise in $\Mcal_s$, see~Corollary~\ref{cor:convexintegration}.

\begin{theorem}[{\bf \cite[Theorem~4]{CoT05}}]\label{theo: Conti Theil}
Let $\Omega\subset \R^2$ be a bounded domain and $\mu\in (0,1)$. Suppose that $F, G\in \Mcal_s$ are rank-one connected with $Fs\neq Gs$ and $N=\mu F +(1-\mu)G\in \Ncal_s$.

Then for every $\delta>0$ there are $h_\delta^0 > 0$ and $\Omega_\delta \subset \Omega$ with $|\Omega\setminus\Omega_\delta| < \delta$ such that the restriction to $\Omega_\delta$ of any simple laminate between the gradients $F$ and $G$ with weights $\mu$ and $1-\mu$ and period $h < h_\delta^0$ can be extended to a finitely piecewise affine function $v_\delta: \Omega \rightarrow \R^2$ with $\nabla v_\delta\in \Ncal_s$ a.e.~in $\Omega$, $v_\delta = Nx$ on $\partial \Omega$, and $\dist(\nabla v_\delta, [F,G]) < \delta$ a.e.~in $\Omega$, where $[F, G]=\{t F+ (1-t)G: t\in [0,1]\}$. 
\end{theorem}

Convex integration methods help to obtain exact solutions to partial differential inclusions.
\begin{theorem}[{\bf \cite[Theorem~1.3]{MuS99}}]\label{thm:convexintegration}
Let $\Mcal \subset \{F \in \R^{2\times 2} : \det F = 1 \}$.  
Suppose that $(U_i)_{i}$ is an in-approximation of $\Mcal$, i.e.,~the sets $U_i$ are open in $\{F\in \R^{2\times 2}:\det F=1\}$ and uniformly bounded, $U_i$ is contained in the rank-one convex hull of $U_{i+1}$ for every $i\in \N$, and $(U_i)_i$ converges to $\Mcal$ in the following sense: if $F_i \in U_i$ for $i\in\N$ and $\abs{F_i - F}\to 0$ as $i\to \infty$, then $F \in \Mcal$. 

Then, for any $F\in U_1$ and any open domain $\Omega \subset \R^2$, there exists $u\in W^{1, \infty}(\Omega;\R^2)$ such that
$\nabla u \in \Mcal$ a.e.~in $\Omega$ and $u = Fx$ on $\partial \Omega$.
\end{theorem}

Combining these two theorems with the explicit construction of Lemma~\ref{lem:construction_rankoneconnections} leads to the following result, cf.~also~\cite{CoT05, Con06, CDK11}.

\begin{corollary}\label{cor:convexintegration}
Let $\Omega\subset \R^2$ be a bounded domain and $N\in \Ncal_s$. If $N\in \Ncal_s\setminus \Mcal_s$, let $F, G\in \Mcal_s$ and $\mu\in (0,1)$ be as in Lemma~\ref{lem:construction_rankoneconnections}, otherwise let $F=G=N\in \Mcal_s$ and $\mu\in (0,1)$.

Then, for every $\delta>0$
there exists $u_\delta \in W^{1, \infty}(\Omega;\R^2)$ and $\Omega_\delta\subset \Omega$ with $\abs{\Omega\setminus \Omega_\delta}<\delta$ such that $u_\delta$ coincides with a simple laminate between $F$ and $G$ with weights $\mu$ and $1-\mu$ and period $h_\delta<\delta$ in $\Omega_\delta$, $\nabla u_\delta \in \Mcal_s$ a.e.~in $\Omega$, $u_\delta=Nx$ on $\partial \Omega$, and
\begin{align}\label{eq_convexintegration}
\abs{\nabla u_\delta m} < \abs{Nm} +\delta \qquad \text{a.e.~in $\Omega$.}
\end{align}
In particular, $\abs{\nabla u_\delta m}= \abs{Nm}$ a.e.~in $\Omega_\delta$, and $\nabla u_\delta \weakly N$ in $L^2(\Omega;\R^{2\times 2})$ as $\delta \to 0$.
\end{corollary}
\begin{proof}
From Theorem \ref{theo: Conti Theil} we obtain for $\delta>0$ the desired set $\Omega_\delta$ along with a finitely piecewise affine function $v_\delta: \Omega \rightarrow \R^2$ that coincides in $\Omega_\delta$ with a simple laminate between the gradients $F$ and $G$ of period $h_\delta<\min\{\delta, h_\delta^0\}$, satisfies $\nabla v_\delta \in \Ncal_s$ a.e.~in $\Omega$, and $v_\delta = Nx$ on $\partial \Omega$. In view of~\eqref{eq19}, $\abs{\nabla v_\delta m}=\abs{Nm}$ a.e.~in $\Omega_\delta$ and 
\begin{align*}
\delta > \dist\big(\nabla v_\delta, [F, G]\big) \geq \min_{t \in [0,1]} \big|\nabla v_\delta m - (t Fm + (1-t)Gm)\big| \geq  |\nabla v_\delta m| - |Nm|
\end{align*}
a.e.~in $\Omega$. Finally, the sought function $u_\delta$ results from a modification of $v_\delta$ in the (finitely many) domains where $\nabla v_\delta\notin\Mcal_s$ by applying Theorem~\ref{thm:convexintegration} with the in-approximation $(U_i^\delta)_i$ of $\Mcal_s\cap \{F\in \R^{2\times 2}: \abs{Fm}< \abs{Nm}+\delta\}$ given by
\begin{align*}
U_i^\delta = \big\{ F \in \R^{2\times2}: \det F = 1, \, 1-2^{-(i-1)}<\abs{Fs} <1,\, \abs{Fm} < \abs{Nm} + \delta \big\}, \quad i\in \N,
\end{align*}
see~\cite[Proof of Lemma~2]{CoT05} for more details.
\end{proof}


\section{Proof of Theorem~\ref{theo:Gammalimit} for $s = e_1$}\label{sec:Gamma_e1}
As indicated in the introduction, in the special case of a horizontal slip direction $s=e_1$, we can prove Theorem~\ref{theo:Gammalimit} in a slightly more general setting, where $W_{\rm soft}$ has an additional linear term that can be interpreted as a dissipative energy contribution. 

More precisely, for a given $\tau\geq 0$ let us replace $W_{\rm soft}$ with
\begin{align*}
W_{\rm soft}^\tau(F) = \begin{cases}\gamma^2 +\tau \abs{\gamma} & \text{if $F=R(\Ibb+\gamma s\otimes m)$, $R\in SO(2)$, $\gamma\in \R$,} \\ \infty& \text{otherwise,} \end{cases}\qquad F\in \R^{2\times 2}.
\end{align*}
Then, $E_\eps$ of~\eqref{def:Eeps} with $\eps>0$ turns into
\begin{align}\label{Eepskappa_representation2}
E_\eps^\tau(u)= \begin{cases}\displaystyle \int_\Omega \gamma^2 +\tau \abs{\gamma} \dd{x} & \text{if $u\in W^{1,2}(\Omega;\R^2)$, $\nabla u=R(\Ibb+\gamma s\otimes m)$} \\ & \text{with $R\in L^\infty(\Omega;SO(2))$, $\gamma \in L^2(\Omega)$, $\gamma=0$ a.e.~in $\eps \Yrig\cap \Omega$,}\\ \infty & \text{otherwise,}\end{cases}
\end{align}
for $u\in L^2_0(\Omega;\R^2)$, cf.~\eqref{Eeps_representation2}.

In this section, we prove the following generalization of Theorem~\ref{theo:Gammalimit} in the case $s=e_1$, assuming that $\Omega$ is simply connected.
\begin{theorem}\label{theo:Gammalimit_e1}
Let $(E_\eps^\tau)_\eps$ as in~\eqref{Eepskappa_representation2} and let the functional $E^\tau: L_0^2(\Omega;\R^2)\to [0, \infty]$ be given by
\begin{align*}
E^\tau (u)=\begin{cases}\displaystyle \frac{1}{\lambda} \int_\Omega \gamma^2 \dd{x} + \tau \int_{\Omega}\abs{\gamma}\dd{x} & \text{if $u\in W^{1,2}(\Omega;\R^2)$, 
$\nabla u=R(\Ibb +\gamma e_1\otimes e_2)$}  \text{~with} \\ &\quad \quad  R\in SO(2), ~\gamma\in L^2(\Omega),  \\ \infty & \text{otherwise.} \end{cases}
\end{align*}
Then, $\Gamma(L^2)\text{-}\lim_{\eps\to 0}E_\eps^\tau =E^\tau$.
Moreover, bounded energy sequences of $(E_\eps^\tau)_\eps$ are relatively compact in $L^2_0(\Omega;\R^2)$.
\end{theorem}

\begin{proof} The proof is divided into three steps.

\textit{Step~1: Compactness.} Let $(\eps_j)_j$ with $\eps_j\to 0$ as $j\to \infty$, and consider $(u_j)_j$ such that $E_{\epsilon_j}^\tau(u_j) < C$ 
for all $j\in \N$. Then, 
$\nabla u_j=R_j(\Ibb+\gamma_j e_1\otimes e_2)$ with $R_j \in L^\infty(\Omega;SO(2))$ and $\gamma_j\in L^2(\Omega)$. 
Since $\abs{\nabla u_j e_1} = 1$ a.e.~in $\Omega$ and $\norm{\nabla u_j e_2}^2_{L^2(\Omega;\R^2)}=\norm{\gamma_j}_{L^2(\Omega)}^2 +\abs{\Omega}$, the observation that $(\gamma_j)_j$ is uniformly bounded in $L^2(\Omega)$ results in
\begin{align}\label{compactnessC}
\norm{\nabla u_j}_{L^2(\Omega;\R^{2\times 2})}\leq C\qquad\text{for all $j\in \N$.}
\end{align}
By Poincar\'{e}'s inequality (recall that $\int_{\Omega} u_j\dd{x}=0$) the sequence $(u_j)_j$ is uniformly bounded in $W^{1,2}(\Omega;\R^2)$. Hence, one may extract a subsequence (not relabeled) of $(u_{j})_j$ that converges weakly in $W^{1, 2}(\Omega;\R^2)$, and also strongly in $L^2(\Omega;\R^2)$ by compact embedding, to a limit function $u\in W^{1,2}(\Omega;\R^2)$. Considering Remark~\ref{rem:rigidity}\,a), we infer from Lemma~\ref{lem:rigidity} that
$\nabla u = R(\Ibb + \gamma e_1 \otimes e_2),$
where $R \in SO(2)$ and $\gamma\in L^2(\Omega)$. 

Moreover, the case $s=e_1$ at hand carries even more information. Indeed, $R_{j} e_1=\nabla u_{j} e_1 \rightharpoonup \nabla ue_1=Re_1$ in $L^2(\Omega;\R^2)$
along with $\abs{R_{j} e_1} =\abs{Re_1} = 1$ entails strong convergence of the rotations $(R_j)_j$, i.e.~(possibly after selection of another subsequence) $R_{j} \rightarrow R$ in $L^2(\Omega;\R^{2\times 2})$, and thus, also 
\begin{align}\label{weakconvergence_gammaeps}
\gamma_{j} \rightharpoonup \gamma\quad \text{in $L^2(\Omega)$. }
\end{align}

\textit{Step~2: Recovery sequence.} Let $u \in W^{1,2}(\Omega; \R^2)\cap L^2_0(\Omega;\R^2)$ with 
\begin{align*} 
\nabla u = R(\mathbb{I} + \gamma e_1\otimes e_2)\quad\text{with $R\in SO(2)$, $\gamma \in L^2(\Omega).$} 
\end{align*}
The main idea for the construction of a recovery sequence is to set $\gamma=0$ in the stiff layers, as the functional $E_\epsilon^\tau$ requires, while compensating with more gliding in the softer layers.

Therefore, for $\eps>0$ we put
\begin{align*}
\gamma_\epsilon = \frac{\gamma}{\lambda} \mathbbm{1}_{\eps\Ysoft\cap \Omega}\in L^2(\Omega).
\end{align*}
Let us assume for the moment that the function $R(\mathbb{I} + \gamma_\epsilon e_1\otimes e_2)\in L^2(\Omega;\R^{2\times 2})$ has a potential, meaning that there exists $u_\epsilon\in W^{1,2}(\Omega;\R^2)$ with
\begin{align}\label{eq23}
\nabla u_\eps = R(\mathbb{I} + \gamma_\epsilon e_1\otimes e_2),
\end{align} 
without loss of generality, we can take $u_\eps\in L_0^2(\Omega;\R^2)$.
By the weak convergence of oscillating periodic functions one has that $\mathbbm{1}_{\eps\Ysoft \cap \Omega} \weaklystar \lambda$ in $L^\infty(\Omega)$, which implies
\begin{align*}
\gamma_\eps \weakly \gamma \quad \text{in $L^2(\Omega)$.}
\end{align*}
Consequently, it follows in view of Poincar\'e's inequality that $u_\epsilon \rightharpoonup u$ in $W^{1,2}(\Omega;\R^2)$, and by compact embedding $u_\epsilon \rightarrow u$ in $L^2(\Omega;\R^2)$. 
Regarding the convergence of energies we argue that
\begin{align*}
\lim_{\eps\to 0}E^\tau_\eps(u_\eps) &=\lim_{\eps\to 0} \int_\Omega \gamma_\eps^2+\tau \abs{\gamma_\eps}\dd{x} = \lim_{\eps\to 0}\int_\Omega \Bigl(\frac{\gamma^2}{\lambda^2}+\frac{\tau}{\lambda}\abs{\gamma} \Bigr)\mathbbm{1}_{\eps\Ysoft\cap \Omega}\dd{x} \\ &= \int_{\Omega} \frac{\gamma^2}{\lambda} + \tau \abs{\gamma} \dd{x} =E^\tau(u). 
\end{align*}

It remains to prove the existence of $u_\eps \in W^{1,2}(\Omega;\R^2)$ such that~\eqref{eq23} holds. 
If $\Omega$ is a cube $Q\subset \R^2$ with sides parallel to the coordinate axes, say $Q=(0, l)^2$ with $l>0$, then 
$\gamma_\eps$ is independent of $x_1$ in view of Remark~\ref{rem:rigidity}\,b) and the orientation of the layers, hence, depending on the context, it can be interpreted as an element in  $L^2(Q)$ or $L^2(0, l)$. Then, for any $a\in \R^2$,
\begin{align}\label{eq40}
u_\eps(x) = Rx+ \Bigl(\int_0^{x_2} \gamma_\eps(t)\dd{t}\Bigr) Re_1 +a, \qquad x\in Q,
\end{align}
satisfies $u_\eps\in W^{1,2}(Q;\R^2)$ with $\nabla u_\eps=R(\Ibb+\gamma_\eps e_1\otimes e_2)$.
To construct $u_\eps$ 
for a general $\Omega$,
we exhaust $\Omega$ successively with shifted, overlapping cubes $Q$, using~\eqref{eq40} with suitably adjusted translation vectors $a$.

\textit{Step~3: Lower bound.} Let $(\eps_j)_j$ with $\eps_j\to 0$ as $j\to \infty$, and $u_j \to u$ in $L^2(\Omega;\R^2)$. We assume without loss of generality that $(u_{j})_j$ is a sequence of uniformly bounded energy for $(E_{\eps_j}^\tau)_j$, so that $(u_{j})_j$ and $u$ satisfy the properties of Step~1.

If $u$ is piecewise affine, the desired liminf inequality then follows directly from Corollary~\ref{cor: estimate lower bound}, as $\frac{1}{\lambda}\abs{\nabla ue_2-(1-\lambda)Re_2}^2-\lambda=\frac{1}{\lambda}\abs{\lambda Re_2+\gamma Re_1}^2-\lambda=\frac{\gamma^2}{\lambda}$, and  from~\eqref{weakconvergence_gammaeps}. 

To prove the statement for general $u$, we perform an approximation argument inspired by the proof of M\"uller's homogenization result in~\cite[Theorem~1.3]{Mul87}. Due to the differential constraints in $E_{\eps}^\tau$, however, the construction of suitable comparison functions is slightly more involved.

Suppose that $\Omega = Q\subset \R^2$ is a cube with sides parallel to the coordinate axes, otherwise we perform the arguments below on any finite union of disjoint cubes contained in $\Omega$ and take the supremum over all these sets, exploiting the fact that the energy density in~\eqref{Eepskappa_representation2} is non-negative.

Accounting for Remark~\ref{rem:rigidity}\,b) allows us to find a sequence of one-dimensional simple functions $(\zeta_k)_k$ (identified with a sequence in $L^2(Q)$ by constant extension in $x_1$-direction) such that
\begin{align}\label{app3}
\zeta_k\to \gamma \qquad \text{in $L^2(Q)$.}
\end{align} 

For $k\in \N$, let $w_k \in W^{1,2}(\Omega;\R^2)\cap L^2_0(\Omega;\R^2)$ with $\nabla w_k = R(\Ibb + \zeta_k e_1 \otimes e_2)$. 
Further, let $(v_j)_j$ and $(v_{k, j})_j$ be the recovery sequences (as constructed in Step~2) for $u$ and $w_k$, respectively.
We define 
\begin{align}\label{representationzkj}
z_{k,j} = u_j - v_j + v_{k,j}, \qquad j,k\in \N,
\end{align}
observing that $z_{k,j} \weakly w_k$ in $W^{1, 2}(\Omega;\R^{2})$ for all $k\in \N$ as $j\to \infty$. 
Moreover, $\nabla z_{k,j}= \nabla u_j$ a.e.~in $\epsilon_j \Yrig \cap \Omega$, so that in particular, $|\nabla z_{k,j} e_2|= 1$ a.e.~in $\epsilon_j\Yrig \cap \Omega$ for all $j,k\in \N$, and
\begin{align*}
\nabla z_{k,j}\mathbbm{1}_{\epsilon_j \Yrig\cap \Omega} \weakly (1-\lambda)R \quad \text{in $L^2(\Omega;\R^{2\times2})$} 
\end{align*}
for $k\in \N$ as $j\to \infty$ by~\eqref{add-on}.
Considering that $w_k$ is piecewise affine, it follows as in the proof of Corollary~\ref{cor: estimate lower bound}  that
\begin{align}\label{est31}
 \liminf_{j\to \infty} \int_{\Omega} |\nabla z_{k,j} e_2|^2 - 1 \dd x \geq \frac{1}{\lambda} \int_{\Omega} \zeta_k^2 \dd x
\end{align}
for all $k\in \N$.
With~\eqref{representationzkj} one obtains
\begin{align*}
E^\tau_{\eps_j}(u_j) & =  \int_{\Omega}\abs{\nabla u_j e_2}^2 -1 +\tau \abs{\gamma_j}\dd{x} \\ &
\geq   \int_{\Omega}\abs{\nabla z_{k,j} e_2}^2 -1 \dd{x} +\tau \int_\Omega \abs{\gamma_j}\dd{x}  
- 2 \norm{\nabla u_j}_{L^2(\Omega;\R^{2\times 2})} \norm{\nabla v_j e_2 - \nabla v_{k,j}e_2}_{L^2(\Omega;\R^2)} 
\end{align*}
for $j, k\in \N$, where by construction $\nabla v_je_2-\nabla v_{k,j}e_2 = \lambda^{-1}(\gamma-\zeta_k)Re_1\mathbbm{1}_{\eps_j\Ysoft\cap \Omega}$, cf.~Step~2. 
From $\norm{\nabla v_j e_2 - \nabla v_{k,j}e_2}_{L^2(\Omega;\R^2)} \leq \frac{1}{\lambda} \norm{\gamma-\zeta_k}_{L^2(\Omega)}$ and the uniform boundedness of $(\nabla u_j)_j$ in $L^2(\Omega;\R^{2\times 2})$ by~\eqref{compactnessC} we infer that
\begin{align*}
E^\tau_{\eps_j}(u_j) \geq  \int_{\Omega}\abs{\nabla z_{k,j} e_2}^2 -1 \dd{x} +\tau \int_\Omega \abs{\gamma_j}\dd{x} 
- \frac{2 C}{\lambda}  \norm{\gamma-\zeta_k}_{L^2(\Omega)} 
\end{align*} 
for $j, k\in \N$.
Passing to the limit $j\to \infty$ in the above estimate yields
\begin{align*}
\liminf_{j\to \infty} E^\tau_{\eps_j}(u_j) \geq \frac{1}{\lambda}\int_{\Omega} \zeta_k^2\dd{x} + \tau \int_{\Omega}\abs{\gamma}\dd{x} - \frac{2C}{\lambda} \norm{\gamma-\zeta_k}_{L^2(\Omega)}.
\end{align*}
Here we have used~\eqref{est31}, as well as~\eqref{weakconvergence_gammaeps}. 
Finally, due to~\eqref{app3}, taking $k\to \infty$ finishes the proof of the liminf inequality.
\end{proof}

\section{Proof of Theorem~\ref{theo:Gammalimit} for $s \neq e_1$}\label{sec:proof_sneqe1}
This section is concerned with the proof of Theorem~\ref{theo:Gammalimit} in the case of an inclined or vertical slip direction, that is $s\neq e_1$.
Due to the strong restrictions on rank-one connections between $SO(2)$ and $\Mcal_{s}$ with normal $e_2$ (see Remark \ref{rem:rank1_connections}), 
the construction of recovery sequences is more involved than for $s=e_1$.

Before giving the detailed arguments, let us briefly discuss different equivalent representations of the limit energy $E$ introduced in~\eqref{def:GammalimitI}. Depending on the context, we will always use the most convenient one without further mentioning. 
For $u\in W^{1,2}(\Omega;\R^2)$ such that $\nabla u = R(\Ibb + \gamma e_1 \otimes e_2)$ with $R\in SO(2)$ and $\gamma\in L^2(\Omega)$ we define with regard to Lemma~\ref{lem: specific laminates in N} the function $N\in L^2(\Omega;\R^{2\times 2})$ by $\lambda N + (1-\lambda) R = \nabla u$, i.e.~$N=R(\Ibb+\tfrac{\gamma}{\lambda} e_1\otimes e_2)$. Then, 
\begin{equation}\label{representationsE}
\begin{aligned}
E(u)&= \frac{s_1^2}{\lambda} \int_\Omega \gamma^2 \dd{x} - 2 s_1 s_2 \int_\Omega \gamma \dd x 
= \int_{\Omega} \frac{1}{\lambda}|\gamma m_2 e_1 +\lambda m|^2 -\lambda \dd x \\
& = \int_{\Omega} \frac{1}{\lambda} \abs{\nabla um - (1-\lambda)Rm}^2 - \lambda \dd x =\lambda \int_\Omega \abs{Nm}^2 -1 \dd{x}.
\end{aligned}
\end{equation}

\begin{proof}[Proof of Theorem~\ref{theo:Gammalimit}] Here again, the proof follows three steps.

\textit{Step~1: Compactness.} The proof of compactness is identical with the beginning of Step~1 in Theorem \ref{theo:Gammalimit_e1} for $\tau=0$, when substituting $e_1$ with $s$ and $e_2$ with $m$.

\textit{Step~2: Recovery sequence.}
Let $u\in W^{1,2}(\Omega;\R^2)\cap L^2_0(\Omega;\R^2)$ such that $\nabla u = R(\Ibb + \gamma e_1 \otimes e_2)$ with $R\in SO(2)$ and $\gamma\in L^2(\Omega)$ be given. 
The idea of the construction is to specify first a sequence of functions with asymptotically optimal energy that are piecewise affine on the layers and whose gradients lie in $\Ncal_s$. Then, to obtain admissible deformations, we approximate these functions in the softer layers with fine simple laminates between gradients in $\Mcal_s$, which requires tools from relaxation theory and convex integration as discussed in Section~\ref{subsec:construction_tools}.

\textit{Step~2a: Auxiliary functions for constant $\gamma$.} 
Let $\gamma \in K_{s, \lambda}$ be constant.
By Lemma~\ref{lem: specific laminates in N} we find $N \in \Ncal_s$ such that
\begin{align}\label{repr_gradient}
\lambda N+(1-\lambda) R = R(\Ibb +\gamma e_1\otimes e_2)
\end{align} 
and $Ne_1=Re_1$, which guarantees the compatibility for constructing laminates between the gradients $R$ and $N$ with $e_2$ the normal on the jump lines of the gradient.
Precisely, we define for $\eps>0$ the function $v_\eps \in W^{1, 2}(\Omega;\R^2)$ with zero mean value characterized by $\nabla v_\eps=\nabla v_1(\eps^{-1}\frarg)$, where $v_1\in W^{1, \infty}_{\rm loc}(\R^2;\R^2)$ is such that
\begin{align}\label{def_v1}
\nabla v_1 =  R \mathbbm{1}_{\Yrig} + N \mathbbm{1}_{\Ysoft}.
\end{align} 

Then, by the weak convergence of highly oscillating functions and~\eqref{repr_gradient},
\begin{align*}
\nabla v_\epsilon \rightharpoonup \lambda N +(1-\lambda)R = \nabla u \quad \text{in $L^2(\Omega;\R^{2\times 2})$}.
\end{align*}
Regarding the energy contribution of the sequence $(v_\eps)_\eps$ it follows that
\begin{align*}
\lim_{\eps\to 0}\int_{\Omega} \abs{\nabla v_\eps m}^2 -1\dd{x}  =  \lim_{\eps\to 0}\int_\Omega (\abs{Nm}^2 -1)\mathbbm{1}_{\epsilon \Ysoft\cap \Omega} \dd x = \lambda  \abs{\Omega} (\abs{Nm}^2 -1) =E(u).
\end{align*}
Notice that  $v_\eps$ is  not admissible for $E_\eps$ if $N\in \Ncal_s\setminus \Mcal_s$.

\textit{Step~2b: Admissible recovery sequence for constant $\gamma$.} 
Next, we modify the construction of Step~2a in the softer layers to obtain admissible functions, while preserving the energy. This is done by approximation with the simple laminates established in Corollary~\ref{cor:convexintegration}, see Figure~\ref{fig:laminates} for illustration.

\begin{figure}
 \begin{tikzpicture}[scale = 1.1]

       \begin{scope}

         \draw[pattern = north east lines] (0,0) rectangle (4.2,2.8);
         \fill[even odd rule, fill = white, very thick] {(2.1,1.4) ellipse (1.9 and 1.2)} (0,0) rectangle (4.2,2.8);
         \draw (2.1,1.4) ellipse (1.9 and 1.2);

         \draw (0,0) rectangle (4.2,4.2);
         \draw[fill=gray!30] (0,2.8) rectangle (4.2,4.2);
         
         \draw (2.1,3.5) node {$R \in SO(2)$}; 
      \draw (6.2, 3.95) node {$N = \mu F + (1-\mu)G $};
         \draw[thick] (3.2,0.9) circle (0.3);
        \draw[thick,dashed](2.988,1.112) -- (4.998,3.102);
        \draw[thick, dashed](3.2,0.6) -- (6.2,0.2);

         \begin{scope}[shift = {(6.2,1.9)}]
         \clip (0,0) circle (1.7);
         \draw[thick, color = gray, fill = white] (0,0) circle (1.7);
         
         \draw (-2,0) -- (2,4);
	 \draw (-2,-1) -- (2,3);
         \draw (-2,-2) -- (2,2);
         \draw (-2,-3) -- (2,1);
         \draw (-2,-4) -- (2,0);

         \draw (-0.75,0.75) node {$F$};
         \draw (-0.25,0.25) node {$G$};
         \draw (0.25,-0.25) node {$F$};
         \draw (0.75,-0.75) node {$G$};
         \end{scope}
         
   	 \draw (-0.4, 3.95) node {$Y$};
         \draw[thick] (6.2,1.9) circle (1.7);
         
         \draw[->, thick] (-1.6,2.4) -- (0.3,2.4);
         \draw[->, thick] (-1.6,0.4) -- (0.3,0.4);
	\draw[very thick] (-1.6, 2.4) -- (-1.6, 2);  
	\draw[very thick] (-1.6, 0.4) -- (-1.6, 0.8);    
         \draw (-1.6,1.6) node {convex};
         \draw (-1.6,1.2) node {integration};
      \end{scope}\clip(-0.7,-0) rectangle (7,4);

    \end{tikzpicture}
  \caption{Construction of admissible deformations by approximation with fine simple laminates in the softer component (illustrated by $\nabla z_\eps$ in the unit cell $Y$). Here, $Ne_1=Re_1$, $N=\mu F+(1-\mu) G$ with $F, G\in \Mcal_s$ and $\mu\in (0,1)$ as in Lemma~\ref{cor:convexintegration}, and the measure of the boundary region of $\Ysoft$ is smaller than $\eps$. }\label{fig:laminates}
    \end{figure}
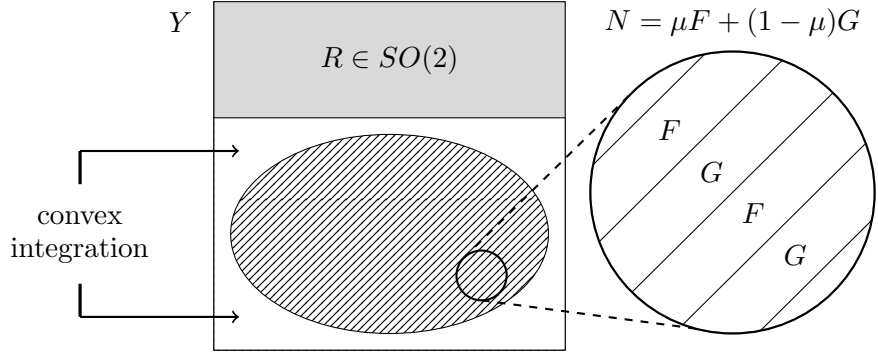

Let $N\in \Ncal_s$ as in Step~2a. For $\eps>0$ and $i\in \Z^2$, let $\varphi_{\eps, i}\in W^{1, \infty}(i+(0,1)\times (0, \lambda);\R^2)$ be a function resulting from Corollary~\ref{cor:convexintegration} applied to $\Omega=i+(0,1)\times (0, \lambda)\subset \R^2$ and $\delta=\eps$. 
We define
\begin{align}\label{eq64}
\varphi_\eps(x) = \sum_{i\in \Z^2} \bigl(\varphi_{\eps, i}(x)-Nx\bigr)\mathbbm{1}_{i+(0,1)\times (0, \lambda)}, \qquad x\in \R^2,
\end{align} 
assuming without loss of generality that $\varphi_\eps \in W^{1, \infty}(\R^2;\R^2)$ is $Y$-periodic. 
With $v_1$ from Step~2a, we set $z_{\eps} = v_1 + \varphi_\eps$. 
Since $\nabla \varphi_\eps \weakly 0$ in $L^2_{\rm loc}(\R^2;\R^{2\times 2})$ as $\eps\to 0$,
\begin{align}\label{convergence9}
\nabla z_{\eps}\weakly \nabla v_1 \quad \text{in $L^2_{\rm loc}(\R^2;\R^{2\times 2})$}.
\end{align}

Defining $u_{\eps}\in W^{1,2}(\Omega;\R^2)$ with mean value zero by 
\begin{align*}
\nabla u_{\eps}= \nabla z_{\eps} \Bigl(\frac{\frarg}{\eps}\Bigr)
\end{align*}
provides admissible functions for $E_\eps$. 
Indeed, by construction one has  $\nabla u_{\eps}=R\in SO(2)$ in $\eps\Yrig\cap \Omega$ and $\nabla u_{\eps}\in \Mcal_s$ a.e.~in $\Omega$ due to the properties of $\nabla \varphi_{\eps,i}$.

In view of~\eqref{convergence9}, a generalization of the classical lemma on weak convergence of highly oscillating sequences (see e.g.~\cite[Theorem~1]{LuW02}) yields 
\begin{align}\label{convergence10}
\nabla u_\eps\weakly \int_{Y} \nabla v_1\dd{y} = \lambda N+(1-\lambda)R =\nabla u \quad \text{in $L^2(\Omega;\R^{2\times 2})$}.
\end{align}

Finally, as $\abs{\nabla z_\eps m}< \abs{Nm} + \eps$ a.e.~in $ \R^2$ in consequence of~\eqref{eq_convexintegration} and $\abs{Rm}=1\leq \abs{Nm}$, we conclude that
\begin{align*}
\limsup_{\eps \rightarrow\infty} E_\eps(u_{\eps}) 
&= \limsup_{\eps\to 0} \int_{\Omega}(\abs{\nabla z_\eps (\eps^{-1}\frarg) m}^2-1)
\mathbbm{1}_{\eps\Ysoft\cap \Omega} \dd{x}\\
& \leq \lim_{\eps\to 0}  \int_{\Omega} (\abs{N m}^2-1) \mathbbm{1}_{\eps\Ysoft \cap \Omega} \dd{x} =  E(u).
\end{align*}

\textit{Step~2c: Localization for piecewise constant $\gamma$.}
Suppose that $\Omega\subset \R^2$ is a cube, say $Q=(0, l)^2$ with $l>0$. 
In this step, the construction of Step~2b is extended to (finitely) piecewise constant $\gamma$.

In view of the independence of $\gamma$ on $x_1$ by~Remark~\ref{rem:rigidity}\,b), we may identify $\gamma\in L^2(\Omega)$ with a one-dimensional simple function
\begin{align*}
\gamma(t)=\sum_{i=1}^n \gamma_i \mathbbm{1}_{(t_{i-1}, t_i)}(t), \qquad t\in (0,l),
\end{align*} 
with $\gamma_i\in K_{s, \lambda}$ for $i=1,\ldots, n$ and $0=t_0 < t_1<\ldots  < t_n=l$. 
Let us denote by $N_i\in \Ncal_s$ the matrices corresponding to $\gamma_i$ according to Lemma~\ref{lem: specific laminates in N}\,\textit{i}).

For $i\in \{1, \ldots, n\}$, let $({u}_{i, \eps})_\eps \subset W^{1,2}((0, l)\times (t_{i-1}, t_i);\R^2)$ be the recovery sequences corresponding to $\gamma_i$ as constructed in Step~2b. We then
define $u_\eps\in W^{1,2}(\Omega;\R^2)$ with vanishing mean value by
\begin{align*}
\nabla u_\eps = R + \sum_{i=1}^n (\nabla u_{i, \eps}-R)\mathbbm{1}_{\eps\Ysoft\cap \Omega}  \mathbbm{1}_{\mathbb{R}\times ( \lceil\eps^{-1 }t_{i-1}\rceil \eps, \lfloor \eps^{-1}t_{i}\rfloor\eps)}.
\end{align*}
Notice that $u_\eps$ is well-defined due to the compatibility between $\nabla u_{i, \eps}$ and $R$ along the jump lines $\R\times \eps\Z$. 
Then,~\eqref{convergence10} leads to
\begin{align*}
\nabla u_\eps \weakly &  \sum_{i=1}^n (\lambda N_i  +(1-\lambda)R)\mathbbm{1}_{[\R\times (t_{i-1}, t_{i})]\cap \Omega} 
=\nabla u\qquad \text{in $L^2(\Omega;\R^{2\times 2})$,}
\end{align*}
and regarding the energy contributions it follows that
\begin{align*}
\lim_{\eps\to 0} \int_\Omega |\nabla u_\epsilon m|^2 -1 \dd x 
& = \lim_{\eps\to 0}  \int_{\Omega} \sum_{i=1}^n(\abs{N_im}^2-1)\mathbbm{1}_{\eps\Ysoft \cap \Omega}\mathbbm{1}_{\mathbb{R}\times ( \lceil\eps^{-1 }t_{i-1}\rceil \eps, \lfloor \eps^{-1}t_{i}\rfloor\eps)} \dd{x} \nonumber \\&=\lambda \int_{\Omega} \sum_{i=1}^n(\abs{N_im}^2-1)\mathbbm{1}_{[\R\times (t_{i-1}, t_i)]\cap \Omega}\dd{x} = E(u).
\end{align*}

To generalize the result to a Lipschitz domain $\Omega$, we exhaust $\Omega$ successively with shifted, overlapping cubes, performing the necessary adaptions of the glued-in laminate constructions as well as the appropriate translations, 
cf.~Step~2 of Theorem~\ref{theo:Gammalimit_e1} for a related argument.

\textit{Step~2d: Approximation and diagonalization for general $\gamma$.} 
For general $\gamma \in L^2(\Omega)$ with $\gamma \in K_{s, \lambda}$ a.e.~in $\Omega$ we use an approximation and diagonalization argument.

Let $(\zeta_k)_k\subset L^2(\Omega)$ be a sequence of simple functions with $\zeta_k\to \gamma$ in $L^2(\Omega)$.
For $k\in \N$, let $w_k \in W^{1,2}(\Omega;\R^2)$ with $\int_{\Omega}w_k\dd{x} = 0$ be defined by $\nabla w_k = R(\Ibb + \zeta_k e_1 \otimes e_2)$. Then,
\begin{align}\label{con13}
\norm{w_k-u}_{L^2(\Omega;\R^2)}\leq c \norm{\nabla w_k-\nabla u}_{L^2(\Omega;\R^{2\times 2})} \leq c \norm{\zeta_k-\gamma}_{L^2(\Omega)},
\end{align} 
and 
\begin{align}\label{con14}
\abs{E(w_{k}) - E(u)} \leq \norm{\zeta_k+\gamma}_{L^2(\Omega)}\norm{\zeta_k-\gamma}_{L^2(\Omega)}\leq c \|\zeta_k-\gamma\|_{L^2(\Omega)}
\end{align}
for all $k\in \N$ with a constant $c=c(\gamma)>0$.
If $(w_{k, \eps})_\eps$ is a recovery sequence for $w_k$ as constructed in Step~2c, then
in particular, 
$\lim_{\eps\to 0}\norm{w_{k,\eps}-w_k}_{L^2(\Omega;\R^2)}=0$, and $\lim_{\eps\to 0} E_{\eps}(w_{k,\eps}) =E(w_k)$ for all $k\in\N$.
Hence, together with~\eqref{con13} and~\eqref{con14},
\begin{align*}
&\lim_{k\rightarrow \infty} \lim_{\eps\to 0} \|w_{k,\eps} - u\|_{L^2(\Omega;\R^2)} + \abs{E_{\eps}(w_{k,\eps}) - E(u)}\\ 
& \qquad\qquad \leq
\lim_{k\to \infty}  \norm{w_k-u}_{L^2(\Omega;\R^2)}+\abs{E(w_k)-E(u)} =0 . 
\end{align*}
From the selection principle by Attouch~\cite[Corollary~1.16]{Att84}, we infer the existence of a diagonal sequence $(u_\eps)_\eps$
with $u_\eps=w_{k(\eps), \eps}$ such that
\begin{align*}
 u_\eps \rightarrow u \quad \text{in $L^2(\Omega; \R^2)$}\qquad \text{and} \qquad E_{\eps}(u_{\eps}) \rightarrow E(u)
\end{align*}
as $\eps\to 0$. Notice that also $u_\eps \weakly u$ in $W^{1,2}(\Omega;\R^2)$ in consideration of Step~1.

\textit{Step~3: Lower bound.} 
Let $(\eps_j)_j$ with $\eps_j\to 0$ as $j\to \infty$. Suppose $(u_j)_j$ is a bounded energy sequence for $(E_{\eps_j})_j$ with $u_j\to u$ in $L^2(\Omega;\R^2)$. 
By Step~1, we know that $u\in W^{1,2}(\Omega;\R^2)$ with $\nabla u = R(\Ibb + \gamma e_1 \otimes e_2)$ for $R\in SO(2)$ and $\gamma \in L^2(\Omega)$. 

If $\gamma$ is piecewise constant, then
\begin{align*}
\liminf_{j\to \infty} E_{\eps_j}(u_j) \geq \int_{\Omega} \frac{1}{\lambda} \abs{\nabla um-(1-\lambda)Rm}^2 -\lambda \dd{x}=E(u)
\end{align*}
as an immediate consequence of Corollary~\ref{cor: estimate lower bound}.

Similarly to Step~3 in the proof of~Theorem~\ref{theo:Gammalimit_e1}, we use approximation to establish the lower bound for general $u$. Here again, we may restrict ourselves to working with the assumption that $\Omega$ is a cube, say $\Omega=Q=(0,l)^2$ for $l>0$.

Let $(\zeta_k)_k\subset L^2(0,l)$ be one-dimensional simple functions (identified with a sequence in $L^2(\Omega)$ by constant extension in $x_1$) of the form $\zeta_k= \sum_{i=1}^{n_k} \zeta_{k,i} \mathbbm{1}_{(t_{k, i-1}, t_{k, i})}$ with nested partitions $0=t_{k, 0}<t_{k,1}< \ldots < t_{k, n_k}=l$ such that $\zeta_k\leq \zeta_{k+1}$ and
\begin{align}\label{est40}
\zeta_k\to \gamma \quad \text{in $L^2(\Omega)$ as $k\to \infty$.}
\end{align}
Moreover, let $w_k\in  W^{1,2}(\Omega;\R^2) \cap L^2_0(\Omega;\R^2)$ be given by $\nabla w_k=R(\Ibb+\zeta_k e_1\otimes e_2)$ for $k\in \N$.

In the following, we aim at finding sequences $(v_{j})_j, (v_{k,j})_j \subset W^{1, 2}(\Omega;\R^2)$ with vanishing mean value such that 
\begin{align}\label{convergence11}
v_{j} \weakly u \quad \text{and}\quad  v_{k,j} \weakly w_k\quad\text{both in $W^{1, 2}(\Omega;\R^2)$ as $j\to \infty$}
\end{align} 
for all $k\in\N$, and 
\begin{align}\label{eq45}
\nabla v_{k,j}= \nabla v_{ j} \quad\text{a.e.~in $\epsilon_j \Yrig \cap \Omega$}
\end{align}
for all $j,k\in \N$. Moreover, we seek to have an estimate of the type
\begin{align}\label{est50}
\norm{\nabla v_{k, j}- \nabla v_{j}}_{L^2(\Omega;\R^{2\times 2})} \leq c\norm{\zeta_k-\gamma}_{L^2(\Omega)}
\end{align}
with $c>0$ independent of $j, k$.

Notice that instead of using recovery sequences for $(v_j)_j$ and $(v_{k,j})_j$, we will choose the piecewise affine functions obtained from Step~2, when skipping Step~2b (where fine laminates are glued in the softer layers). Indeed, the lack of admissibility does not cause any issues here. The advantage, though, is that due to their simpler structure, these functions are easier to compare in the sense of~\eqref{est50}. Recall that the full recovery sequences of Step~2d involve regions resulting from convex integration, where the  functions are not explicitly known and therefore hard to control.

Precisely, for $j, k\in \N$ we define $v_{k,j}$ by
\begin{align*}
\nabla v_{k, j} = R + \sum_{i=1}^{n_k} (N_{k, i}-R) \mathbbm{1}_{\eps_j\Ysoft\cap \Omega} \mathbbm{1}_{\R\times (\lceil \eps_j^{-1}t_{k,(i-1)}\rceil\eps_j, \lfloor\eps_j^{-1}t_{k,i}\rfloor\eps_j)}
\end{align*}
with $N_{k, i}\in \Ncal_s$ corresponding to $\zeta_{k, i}$ in the sense of Lemma~\ref{lem: specific laminates in N}\,\textit{i}),
while $(v_{ j})_j$ results from 
a diagonalization argument as in Step~2d, i.e.~$v_{ j} = v_{k(j), j}$ for $j\in \N$. 
Hence,~\eqref{convergence11} and~\eqref{eq45} are satisfied. Regarding~\eqref{est50}, we argue that for 
$j, k, K\in \N$, 
\begin{align*}
\norm{\nabla v_{k,j} - \nabla v_{K, j}}_{L^2(\Omega;\R^{2\times 2})} 
&\leq  \normB{ \sum_{i=1}^{n_{k}} \sum_{h=1}^{n_K} ({N}_{k, i} - N_{K, h}) \mathbbm{1}_{\R\times (t_{K, h-1}, t_{K, h})}\mathbbm{1}_{\R\times(t_{k, i-1}, t_{k, i})}}_{L^2(\Omega;\R^{2\times 2})}\\
&= \frac{1}{\lambda} \normB{ \sum_{i=1}^{n_{k}} \sum_{h=1}^{n_K}  ({\zeta}_{k, i} - \zeta_{K, h})\mathbbm{1}_{\R\times (t_{K, h-1}, t_{K, h})}\mathbbm{1}_{\R\times (t_{k, i-1}, t_{k, i})}}_{L^2(\Omega)} \\
&= \frac{1}{\lambda}\norm{\zeta_{k}-\zeta_{K}}_{L^2(\Omega)},
\end{align*}
where we have used~\eqref{Ngamma}. 
Thus,
\begin{align*}
\norm{\nabla v_{k,j} - \nabla v_j}_{L^2(\Omega;\R^{2\times 2})} \leq \sup_{K\in \N} \norm{\nabla v_{k,j} - \nabla v_{K, j}}_{L^2(\Omega;\R^{2\times 2})}
\leq \frac{1}{\lambda}\norm{\zeta_k-\gamma}_{L^2(\Omega)},
\end{align*}
which yields~\eqref{est50}.

Now we set
\begin{align*}
z_{k,j} = u_j - v_j + v_{k,j}
\end{align*}
for $j, k\in \N$. Due to~\eqref{eq45}, it holds that $|\nabla z_{k,j} e_2|= 1$ a.e.~in $\epsilon_j\Yrig \cap \Omega$, as well as
\begin{align*}
\nabla z_{k,j}\mathbbm{1}_{\epsilon_j \Yrig}  
\weakly (1-\lambda)R \quad \text{in $L^2(\Omega;\R^{2\times2})$} 
\end{align*}
as $j\to \infty$ for $k\in \N$, cf.~Proposition~\ref{prop:rigidity}.
Since also $z_{k,j}\weakly w_k$ in $W^{1, 2}(\Omega;\R^2)$ for $j\to \infty$ by~\eqref{convergence11} with $w_k$ piecewise affine, we argue as in the proof of Corollary~\ref{cor: estimate lower bound} to derive
\begin{align*}
 \liminf_{j\to \infty} \int_{\Omega} |\nabla z_{k,j} m|^2 - 1 \dd x \geq E(w_k)
\end{align*}
for all $k\in \N$.
Along with~\eqref{est50} and~\eqref{est40}, we finally conclude that
\begin{align*}
\liminf_{j\to \infty} E_{\eps_j}(u_j) \geq \lim_{k\to \infty} E(w_k)  -  c\norm{\zeta_k-\gamma}_{L^2(\Omega)}= E(u),
\end{align*}
in analogy to Step~3 in the proof of Theorem~\ref{theo:Gammalimit_e1}.
\end{proof}

\begin{remark}
It may be more intuitive from the point of view of applications - yet technically more elaborate - to replace the recovery sequence obtained in Step~2b for the affine case by optimal deformations showing ``non-stop" simple laminates throughout the softer layers. For this construction, just dispense with the adjustment of the affine boundary conditions along the vertical edges of the unit cell, and instead refine and shift the laminate appropriately to guarantee $Y$-periodicity. 
\end{remark}

\section{Comparison with the (multi)cell formula}\label{sec:cell_formula}
It is a well-known result in the theory of periodic homogenization of integral functionals with standard growth that  the integrand of the effective limit functional is characterized by a multicell formula, or, in the convex case, by a cell formula, see e.g.~\cite{Mul87},~\cite{Mar78}.
In this final section, we show that the same is true for the homogenization result in Theorem~\ref{theo:Gammalimit}, where extended-valued functionals appear.

Recalling $W$ defined in~\eqref{def_W}, we consider the multicell formula
\begin{align*}
 W_\#(F)= \inf_{k\in \N} \inf_{\psi\in W^{1,2}_\#(kY;\R^2)} \frac{1}{k^2}\int_{kY} W(y, F+\nabla \psi(y))\dd{y}
\end{align*}
for $F\in \R^{2\times 2}$, or equivalently, by a change of variables,
\begin{align}\label{homformula}
W_\#(F)=\liminf_{k\to \infty} \inf_{\psi\in W^{1,2}_\#(Y;\R^2)}\int_Y W(ky, F+\nabla\psi(y))\dd{y},
\end{align}
as well as the cell formula 
\begin{align*}
W_{\rm cell}(F)= \inf_{\psi\in W^{1,2}_\#(Y;\R^2)} \int_{Y} W(y, F+\nabla \psi(y))\dd{y}.
\end{align*}

Moreover, let us denote by $W_{\rm hom}$ the density of the limit energy $E$ in \eqref{def:GammalimitI}, i.e.~ 
\begin{align}\label{representation3}
E(u)=\int_{\Omega} W_{\rm hom}(\nabla u)\dd{x},\qquad u\in W^{1,2}(\Omega;\R^2),
\end{align}
where
\begin{align*}
W_{\rm hom} (F) = \begin{cases}\frac{1}{\lambda}\abs{Fm-(1-\lambda)Rm}^2-\lambda & \text{if $F= R(\Ibb+\gamma e_1\otimes e_2)$, $\gamma\in K_{s, \lambda}$,} \\
\infty & \text{otherwise,}\end{cases}\quad F\in \R^{2\times 2}.
\end{align*}
This alternative representation of $E$ follows from a straightforward calculation, see~\eqref{representationsE}. 

Before focusing on the relation between $W_{\rm hom}$, $W_{\#}$, and $W_{\rm cell}$, we prove the following auxiliary result.

\begin{lemma}\label{lem:rigid_periodic}
Let $F\in \R^{2\times 2}$ and $\psi\in W^{1, 2}_\#(Y;\R^2)$.

$i)$ If $F+\nabla \psi \in \Mcal_s$ a.e.~in $Y$ and $F+\nabla \psi\in SO(2)$ a.e.~in $\Yrig$, then $F\in  \Mcal_{e_1}\cap \Ncal_s$. 

$ii)$ If $F=R(\Ibb+\gamma e_1\otimes e_2)$ with $R\in SO(2)$ and $\gamma\in \R$ and $F+\nabla \psi\in \Mcal_{e_1}$ a.e.~in $Y$, then  	
\begin{align*}
 		F+\nabla \psi =R(\Ibb + \zeta e_1\otimes e_2),
	\end{align*} 
	where $\zeta\in L^2(Y)$ with $\int_Y \zeta \dd{y}=\gamma$.

\end{lemma}

\begin{proof} 

For $i)$,  let $R\in L^\infty(Y;SO(2))$ and $\zeta\in L^2(Y)$ such that $F+\nabla \psi=R(\Ibb+\zeta s\otimes m)$. 
Then, using that $\psi$ is $Y$-periodic, we obtain
\begin{align*}
\abs{Fs} = \absB{\int_Y Fs +\nabla \psi s \dd{y} }=\absB{ \int_Y Rs\dd{y}} \leq 1.
\end{align*}
As the map $\det: F\mapsto \det F$ is quasiaffine, i.e.~$\det$ and $-\det$ are both quasiconvex (see~e.g.~\cite{Dac08} for an introduction to generalized notions of convexity), it follows that 
\begin{align*}
\det F= \int_{Y}\det(F+\nabla \psi)\dd{y} = \int_Y \det R\cdot \det (\Ibb +\zeta s\otimes m) \dd{y} =1.
\end{align*}
Hence, $F\in \Ncal_s$.

To prove $F\in \Mcal_{e_1}$, or rather $\abs{Fe_1}= 1$, we exploit that by rigidity (see~Lemma~\ref{lem:rigidity}), there exists $Q\in SO(2)$ such that $F+\nabla \psi=Q$ a.e.~in $\Yrig$. The periodicity of $\psi$ in $y_1$ then leads to 
\begin{align*}
Fe_1= \dashint_{\Yrig} Fe_1+ \partial_1 \psi \dd{y} = \dashint_{\Yrig} Fe_1+ \nabla \psi e_1\dd{y} = Qe_1,
\end{align*} which entails
$\abs{Fe_1}=1$.

As regards $ii)$, since $F+\nabla\psi \in \Mcal_{e_1}$ a.e. in $Y$, we find $Q \in L^\infty(Y; SO(2))$ and $\zeta \in L^2(Y)$ such that 
		\begin{align*}
			F+\nabla \psi = Q(\Ibb +\zeta e_1 \otimes e_2) \quad \text{in $Y$}.
		\end{align*}
		By the periodicity of $\psi$ in $y_1$,
		\begin{align*}
				\int_Y Qe_1 \dd{y} = Fe_1 + \int_Y \nabla \psi e_1\dd{y} = Re_1, 
		\end{align*}
		which, owing to $\abs{Qe_1} = 1$ a.e.~in $Y$, implies $Q=R$. 
		On the other hand, we derive from the periodicity of $\psi$ in $y_2$ that  
		\begin{align*}
			\Bigl(\int_Y \zeta \dd{y}\Bigr) Re_1=\int_Y \zeta Q e_1\dd{y} = Fe_2 + \int_Y \nabla \psi e_2\dd{y} - \int_Y Qe_2\dd{y} = \gamma Re_1.
		\end{align*}
		This finishes the proof.
\end{proof}

As indicated above, the multicell and cell formula for $W$ both coincide with the homogenized integrand $W_{\rm hom}$.
\begin{proposition}[Characterization of the (multi)cell formula]
With the definitions above it holds that $W_{\rm hom} = W_\# = W_{\rm cell}$.
\end{proposition}

\begin{proof}
Since trivially $W_\#\leq W_{\rm cell}$, it is enough to prove $W_\#\geq W_{\rm hom}$ and  $W_{\rm hom} \geq W_{\rm cell}$.
From Lemma~\ref{lem:rigid_periodic}\,\textit{i}) we infer that $W_{\#}(F)=W_{\rm cell}(F)=\infty$ for $F\notin \Mcal_{e_1}\cap \Ncal_s\supset \{F\in \R^{2\times 2}: W_{\rm hom}(F) <\infty\}$, cf.~\eqref{char_Me1capNs}.
Therefore, in the following, we can restrict ourselves to $F\in \Mcal_{e_1}\cap \Ncal_s$.

\textit{Step~1: $W_\#(F)\geq W_{\rm hom}(F)$.}
Without loss of generality, assume that $W_\#(F)<\infty$. According to the definition of $W_\#$ in~\eqref{homformula} there exists a sequence $(\psi_k)_k\subset W^{1,2}_\#(Y;\R^2)$ with 
\begin{align}\label{eq60}
W_\#(F) = \liminf_{k\to \infty} \int_{Y} W(ky, F+\nabla \psi_k(y))\dd{y}.
\end{align}

Imitating the proofs of Theorem~\ref{theo:Gammalimit} in Sections~\ref{sec:Gamma_e1} and~\ref{sec:proof_sneqe1} regarding compactness (Steps~1) and the lower bound (Steps~3) with $\Omega=Y$, we obtain that there is an subsequence (not relabeled) of $(\psi_k)_k$ and $u\in W^{1,2}(Y;\R^2)$ with $\nabla u=F+\nabla \psi$ for some $\psi\in W^{1,2}_\#(Y;\R^2)$
such that 
\begin{align*}
F+\nabla \psi_k \weakly \nabla u = F+\nabla \psi \quad \text{in $L^2(Y;\R^{2\times 2})$},
\end{align*}
and
\begin{align}\label{eq61}
W_\#(F)\geq E(u)
= \int_Y W_{\rm hom}(F+\nabla \psi)\dd{y},
\end{align} 
in view of~\eqref{eq60} and~\eqref{representation3}.
By the definition of $W_{\rm hom}$, it follows that $F+\nabla \psi\in \Mcal_{e_1}$ a.e.~in $Y$. Along with the assumption $F\in \Mcal_{e_1}$ (precisely, $F=R(\Ibb+\gamma e_1\otimes e_2)$ with $R\in SO(2)$ and $\gamma\in \R$), Lemma~\ref{lem:rigid_periodic}\,\textit{ii}) in conjunction with Jensen's inequality 
yields 
\begin{align}\label{eq62}
\int_Y W_{\rm hom}(F+\nabla \psi)\dd{y} \geq \frac{1}{\lambda}\absB{Fm + \int_Y\nabla \psi m \dd{y} -(1-\lambda)Rm}^2-\lambda =W_{\rm hom}(F).
\end{align}
Joining~\eqref{eq61} and~\eqref{eq62} finishes the proof of Step~1.

\textit{Step~2: $W_{\rm cell}(F)\leq W_{\rm hom}(F)$.} To show the reverse inequality, we build on the construction of recovery sequences for affine limit functions (Step~2) in the proofs of Theorem~\ref{theo:Gammalimit}, again with $\Omega=Y$, see Sections~\ref{sec:Gamma_e1} and~\ref{sec:proof_sneqe1}. Accordingly, there is a sequence $(u_k)_k\subset W^{1,2}(Y;\R^2)$ such that $\nabla u_k\weakly F$ in $L^2(Y;\R^{2\times 2})$ and 
\begin{align*}
W_{\rm hom}(F) = \lim_{k\to \infty} \int_Y W(ky, \nabla u_k(y))\dd{y}.
\end{align*}

By construction, the functions $\psi_k$ defined by $\psi_k(y)=u_k(y)-Fy$ for $y\in Y$ are $k^{-1}Y$-periodic (see~\eqref{eq23}, and \eqref{eq64}, \eqref{def_v1}; indeed, $v_1-F\frarg$ is $Y$-periodic), so that 
\begin{align*}
W_{\rm hom}(F)= \lim_{k\to \infty} \int_Y W(ky, F+ \nabla \psi_k(y))\dd{y} =   \lim_{k\to \infty}
\int_{Y} W(y, F+\nabla \psi_k(k^{-1}y)) \geq W_{\rm cell}(F),
\end{align*}
as stated.
\end{proof}

\section*{Acknowledgements}
The authors would like to thank Georg Dolzmann for his valuable comments on a preliminary version of the manuscript.
This work was partially supported by the Deutsche Forschungsgemeinschaft through the Forschergruppe 797 ``Analysis and computation of microstructure in finite plasticity'', project DO 633/2-1.


\bibliographystyle{abbrv}
\bibliography{Homogenization}
\end{document}